\documentclass[a4paper,DIV=11]{scrartcl}

\usepackage{cite}
\usepackage{amsmath,relsize,amssymb,amsfonts}
\usepackage{amsthm}
\usepackage{algorithmic}
\usepackage{graphicx}
\usepackage{kbordermatrix}
\usepackage{authblk}
\usepackage{textcomp}
\usepackage{url}
\usepackage[dvipsnames]{xcolor}
\usepackage{ulem}

\usepackage[pdffitwindow=false,
            plainpages=false,
            pdfpagelabels=true,
            pdfpagemode=UseOutlines,
            pdfpagelayout=SinglePage,
            bookmarks=false,
            colorlinks=true,
            hyperfootnotes=false,
            linkcolor=blue,
            citecolor=blue]{hyperref}

\newtheorem{theorem}{Theorem}
\newtheorem{corollary}[theorem]{Corollary}
\newtheorem{lemma}[theorem]{Lemma}
\newtheorem{proposition}[theorem]{Proposition}
\newtheorem{assumption}[theorem]{Assumption}
\newtheorem{definition}[theorem]{Definition}
\newtheorem{remark}[theorem]{Remark}
\newtheorem*{myproof}{Proof}

\newcommand{\Lmu}{L_\mu^2(\mathbb{X})}

\newcommand\restrict[1]{\raisebox{-.5ex}{$|$}_{#1}}
\newcommand{\LV}{\mathcal{L}_\mathbb{V}}

\newcommand{\tL}{\tilde{\mathcal{L}}_m}
\newcommand{\tC}{\tilde{C}_m}
\newcommand{\tA}{\tilde{A}_m}

\newcommand{\braces}[1]{\textup{(}#1\textup{)}}
\newcommand{\rmref}[1]{\textup{\ref{#1}}}
\newcommand{\R}{\ensuremath{\mathbb R}}    % Reelle Zahlen
\newcommand{\N}{\ensuremath{\mathbb N}}    % Nat"urliche Zahlen
\newcommand{\calL}{\mathcal L}

\newcommand{\veps}{\varepsilon}

\newcommand{\linspan}{\operatorname{span}}

\newcommand{\ol}{\overline}

\newcommand{\wt}{\widetilde}

\providecommand{\floor}[1]{\left \lfloor #1 \right \rfloor }
\providecommand{\ceil}[1]{\left \lceil #1 \right \rceil }

\newcommand{\XX}{\operatorname{\mathbb{X}}}
\newcommand{\MS}[1]{\color{black}#1\color{black}}

\title{Finite-data error bounds for Koopman-based prediction and control\footnote{F.\ Philipp was funded by the Carl Zeiss Foundation within the project \textit{DeepTurb---Deep Learning in und von Turbulenz}. M.\ Schaller was funded by the DFG (project numbers 289034702 and 430154635). K.\ Worthmann gratefully acknowledges funding by the German Research Foundation (DFG; grant WO\ 2056/6-1, project number 406141926). }}
\author[1,2]{Feliks N\"{u}ske}
\author[1]{Sebastian Peitz}
\author[3]{Friedrich Philipp}
\author[3]{Manuel Schaller}
\author[3]{Karl Worthmann}
\affil[1]{Paderborn University, Germany, sebastian.peitz@upb.de}
\affil[2]{Max Planck Institute for Dynamics of Complex Technical Systems, Magdeburg, Germany, nueske@mpi-magdeburg.mpg.de}
\affil[3]{Technische Universit\"at Ilmemau, Germany, \{friedrich.philipp, manuel.schaller, karl.worthmann\}@tu-ilmenau.de}

\begin{document}

\maketitle

\begin{abstract}%
	\noindent The Koopman operator has become an essential tool for data-driven approximation of dynamical (control) systems, e.g., via extended dynamic mode decomposition. Despite its popularity, convergence results and, in particular, error bounds are still %quite 
	scarce. In this paper, we 
	derive probabilistic bounds for the approximation error and the %(multi-step) 
	prediction error depending on the number of training data points; for both ordinary and stochastic differential equations \MS{while using either ergodic trajectories or i.i.d.\ samples}. %
	\MS{We illustrate these bounds by means of an example with the Ornstein-Uhlenbeck process. 
	Moreover, we extend our analysis to (stochastic) nonlinear control-affine systems.
	We prove error estimates for a previously proposed approach that exploits the linearity of the Koopman generator to obtain a bilinear surrogate control system and, thus, circumvents
	the curse of dimensionality since the system is not autonomized by augmenting the state by the control inputs. 
	To the best of our knowledge, this is the first finite-data error analysis in the stochastic and/or control setting. 
	Finally, we demonstrate the effectiveness of the bilinear approach by comparing it with state-of-the-art techniques showing its superiority whenever state and control are coupled.} %
\end{abstract}

%\begin{IEEEkeywords}
%	Approximation error, %
%	dynamic mode decomposition, %
%	finite data, 
%	Koopman operator, %
%	nonlinear control, % systems, %nonlinear dynamical systems,  %
%	%estimation error, 
%	%error analysis
%	prediction error, 
%	stochastic systems. %stochastic processes
%	%Enter key words or phrases in alphabetical order, separated by commas. List: \underline {http://www.ieee.org/organizations/pubs/ani\_prod/keywrd98.txt}
%\end{IEEEkeywords}

\section{Introduction}
\label{sec:introduction}
The Koopman framework~\cite{Koo31} is the operator-theoretic basis for a wide range of data-driven methodologies to predict the evolution of nonlinear dynamical systems using linear techniques, see, e.g., \cite{Mez05,RMB+09} or the recent %extensive 
survey~\cite{BBKK21} and the references therein. The underlying concept is that observables, which may also be understood as outputs from the systems-and-control perspective, can be propagated forward in time using the linear yet infinite-dimensional Koopman operator or its generator, instead of simulating the nonlinear system and %successively 
evaluating the observable functions. % at the consecutive time instant. 
Its recent success %story 
is closely linked to numerically tractable approximation techniques like extended Dynamic Mode Decomposition (eDMD), see, e.g., \cite{WILLIAMS2015,KLUS2016b,KM18b,KNH20} %and the references therein 
for %both 
numerical techniques and %as well as 
convergence results.

	While the Koopman framework is well established, % and widely applied, 
	approximation results are typically only established in the infinite-data limit, i.e., if \textit{sufficient} data is available. %
	Recently, Lu and Tartakovsky~\cite{LuTart20} %presented an interesting application of the Koopman framework to derive rigorous 
	discussed error bounds w.r.t.\ DMD invoking the seminal work~\cite{KM18b} by Korda and Mezi\'{c}. %
	While the authors numerically demonstrate the effectiveness of their approach even for nonlinear parabolic Partial Differential Equations (PDEs), see also their extension% to the inhomogeneous setting
	~\cite{LuTart21}, there remains a significant gap from a more theoretical point of view since %it is assumed that 
	the approximation error %w.r.t.\ the Koopman semigroup 
	is assumed to be zero for \textit{finite data}, see~\cite[Remark~3.1]{LuTart20}. %Nevertheless, for a limited class of systems based on comparatively restrictive stability assumptions. this claim is approximately/asymptotically true. 
	Mamakoukas and coworkers~\cite{MCTM21} mimick a Taylor-series expansion based on a particular set of observables to approximate the system dynamics of an Ordinary Differential Equation (ODE). This work may be understood as a promising approach to incorporate (local) knowledge on the system dynamics in the Koopman framework. %setting as also highlighted in the respective simulations. 
	However, a bound on the prediction error in terms of data is not deduced. Error bounds for Koopman eigenvalues in terms of the finite-data estimation error were derived in~\cite{Webber2021}, but the estimation error itself was not quantified. In~\cite{Mollenhauer2020}, concentration inequalities were %then 
	applied %in 
	to bound the estimation error for the co-variance and cross-covariance operators involved in Koopman estimation. %
\MS{In the exhaustive preprint~\cite{kurdila2018koopman}, the authors treat the projection error for different approximation spaces such as, e.g., reproducing kernel Hilbert spaces and wavelets. The estimation error is also discussed briefly in Section 8.5. In \cite{ZhanZuaz21}, besides providing a finite-data error bound on the approximation of the Koopman operator in the context of ODEs, the authors estimate the projection error by means of finite-element analysis. %
In conclusion, to the best of our knowledge\footnote{We are already referring to two authoritative references on preprint servers supporting our claim that finite-data error bounds are still missing; thanks to one of the unknown referees for drawing our attention to the still unpublished work~\cite{kurdila2018koopman}.}, \cite{ZhanZuaz21, kurdila2018koopman} are the only works providing rigorous \MS{error bounds for Koopman-based approximations }of a dynamical system governed by a nonlinear ODE. }

	In this paper, we rigorously derive probabilistic bounds on the approximation error (or finite-data estimation error) and the (multi-step) prediction error %error bounds for one-step and multi-step predictions 
	for nonlinear Stochastic Differential Equations (SDEs). \MS{This, of course, also includes nonlinear ODEs. }%, as they can be seen as a particular case of the class of SDEs treated in this paper. } %
	The deduced bounds on the approximation error and prediction accuracy explicitly depend on the number of data points used in eDMD. %; both for ordinary and for stochastic differential equations. %
	To this end, besides using mass concentration inequalities and a numerical error analysis to deal with the error propagation in time, %
	we employ substantially different techniques in comparison to~\cite{ZhanZuaz21} to provide an additional alternative assumption based on ergodic sampling tailored to \MS{stationary } SDEs. \MS{Further, we illustrate the error bounds %in relation to the real error 
	for the Ornstein-Uhlenbeck process.}  %avoid the assumption of identically and independently distributed data. This also allows to use data sampled from a single trajectory for ergodic systems.
	%Moreover, we sketch the extension to the setting that directly involves the Koopman operator. %

	W.r.t.\ the application of Koopman theory in control, a lot of research has been invested over the past years, beginning with the popular \textit{DMD with control} \cite{Proctor2016}, which was later used in Model Predictive Control (MPC) \cite{KM18a}. Another popular method is to use a coordinate transformation into Koopman eigenfunctions \cite{KKB21} or the already mentioned component-wise Taylor series expansion~\cite{MCTM21}. %, which is also applicable for control systems. 
	In \cite{LuShin2020}, the prediction error of the method proposed in \cite{Proctor2016} was estimated using the convergence result of \cite{KM18b}. However, the result %established there 
	is of purely asymptotic nature, i.e., it does not state a convergence rate in terms of data points.
	All approaches mentioned until now yield linear surrogate models of the form $Ax+Bu$, i.e.\ the control enters linearly. %, but they only work for systems where the control enters linearly. 
	For general control-affine systems, numerical simulation studies indicate that bilinear surrogate models %of the form $Ax+(Bx)u$ 
	are better suited, see~\cite{GP17,Peitz2020,BFV21,PB21}. 
%Finally, for general nonlinear systems, transformations into switched systems are a viable and data-efficient approach \cite{PK19}, which even allows us to exploit error bounds for autonomous systems \cite{PB21}.
	The technique proposed in~\cite{PK19,Peitz2020} constructs its surrogate model from $n_c+1$ autonomous Koopman operators, where $n_c$ is the control dimension. %
The key feature is that the state-space dimension is not augmented by the number of control inputs, which counteracts the curse of dimensionality in comparison to the more widespread approach introduced in~\cite{KM18a}. Compared to \cite{Peitz2020}, we present a detailed analysis of the accuracy regarding both the dictionary size as well as the amount of training data. Even though the bound is rather coarse on the operator level, we demonstrate that it correctly captures the qualitative behavior.
	In this context, we provide a probabilistic bound on the \MS{approximation error of the projected Koopman generator, the projected Koopman semigroup and the respective trajectories. } %bound for the approximation of the Koopman generator. %step based on the generator of the respective Koopman semigroup 
To this end, we extend our results %employ the results presented in the first part of this paper and generalize them 
towards nonlinear control systems. Besides a rigorous %error 
bound on the approximation error, we present estimates on the (auto-regressive) prediction accuracy, % in an auto-regressive manner, 
i.e.\ in an open-loop prediction (without feedback). % based on error measurements. % in between. 
This allows for a direct application of our results in MPC. %; both with and without stabilizing terminal conditions, see, e.g., .
%, e.g.\ bilinearly, coupled. %, i.e.\ whenever the impact of the control depends on the current state, 

The paper is structured as follows. Firstly, in Section~\ref{sec:SDE}, we deduce a rigorous bound on the approximation error for nonlinear SDEs. Then, we extend our analysis to nonlinear control-affine systems in Section~\ref{sec:control}. In Section~\ref{sec:examples}, two numerical simulation studies for the Ornstein-Uhlenbeck system (SDE) and the controlled Duffing equation (nonlinear control-affine system) are presented before conclusions are drawn in Section~\ref{sec:conclusion}. %Moreover, the latter also contains a brief outlook on potential future research directions.

\section{Finite-data bounds on the approximation error: nonlinear SDEs} %systems}
\label{sec:SDE}

\noindent
In this section, we analyze the approximation quality of extended Dynamic Mode Decomposition (eDMD) with finitely-many data points for the \MS{finite-dimensional } stochastic differential equation
\begin{align}
\label{e:SDE}
\tag{SDE}
\text{d}X_t = F(X_t)\,\text{d}t + \sigma(X_t) \,\text{d}W_t,
\end{align}
where $X_t \in \mathbb{X}\subset \mathbb{R}^d$ is the state, $F : \mathbb{X} \to \mathbb{R}^d$ is the drift vector field, $\sigma : \mathbb{X} \to \mathbb{R}^{d\times d}$ is the diffusion matrix field, and $W_t$ is a $d$-dimensional Brownian motion. We assume that $F, \, \sigma$ satisfy standard Lipschitz properties to ensure global existence of solutions to~\eqref{e:SDE}, see the textbook \cite{Oksendal2013} for an introduction to this class of systems. We stress that the deterministic case is included by simply setting $\sigma \equiv 0$, leading to the ordinary differential equation
\begin{align*}
\tfrac{\text{d}}{\text{d}t} x(t)= F(x(t)).
\end{align*}
The state space is assumed to be a measure space $(\mathbb{X}, \Sigma_{\XX}, \mu)$ with Borel $\sigma$-algebra $\Sigma_{\XX}$ and probability measure $\mu$. In case of an ODE, the set $\mathbb{X}$ is often assumed to be compact and forward-invariant and the probability measure is the standard Lebesgue measure, cf.\ \cite{ZhanZuaz21}.

\begin{definition}[Koopman operator]
\label{d:koopman}
Let $X_t$ satisfy \eqref{e:SDE} for $t \geq 0$. The Koopman operator semigroup associated with \eqref{e:SDE} is defined by
\begin{align*}
\mathcal{K}^t f(x_0) = \mathbb{E}^{x_0}[f(X_t)] = E[f(X_t)|X_0 = x_0]
\end{align*}
for all bounded measurable functions $f$.
\end{definition}
\noindent In case of ergodic sampling, that is, obtaining data points from a single long trajectory, we will assume invariance of the measure~$\mu$ w.r.t.\ the stochastic process~$X_t$. 
\begin{definition}[Invariant measure with positive density]\label{def:invariant_measure}
	A probability measure  $\mu$ is called invariant if it satisfies
	\[ 
		\int_{\XX} \mathcal{K}^t f \,\mathrm{d}\mu = \int_{\XX} f \,\mathrm{d}\mu
	\]
	for all bounded measurable functions~$f$ and all $t\ge 0$. Further, $\mu$ has an everywhere positive density $\rho:\XX \to \mathbb{R}$ if  $\mu (A) = \int_A \rho(x) \,\mathrm{d}x$ holds for all $A\in\Sigma_{\XX}$.
\end{definition}

\noindent We can now formulate our assumption on the underlying dynamics.
\begin{assumption}\label{as:dyn}
	Let either of the following hold:
	\begin{itemize}
		\item[(a)] The set $\mathbb{X}$ is compact and forward invariant $(\forall\,x^0 \in \mathbb{X}: \mathbb{P}^{x_0}(X_t \in \mathbb{X})=1$ for all $t \geq 0)$ and $\mu$ is the normalized Lebesgue measure. Moreover, the Koopman operator can be extended to a strongly continuous semigroup on the Hilbert space $L^2_\mu(\mathbb{X})$.
		\item[(b)] The probability measure is an invariant measure in the sense of Definition~\ref{def:invariant_measure}. 
	\end{itemize}
\end{assumption}
\noindent We briefly comment on this assumption and first note that forward invariance of $\mathbb{X}$ can be weakened, if one is only interested in estimates for states contained in $\mathbb{X}$, see also \textup{\cite[Section 3.2]{ZhanZuaz21}}. Moreover, if the dynamics obey an ODE, it was shown that the Koopman operator can indeed be extended to a strongly continuous semigroup on $L^2_\mu(\XX)$, see also~\cite{ZhanZuaz21}. Second, the assumption of invariance of the underlying probability measure is satisfied for a broad class of SDEs, see e.g. \cite{Risken1996}. It can be checked that $\mu$ is then invariant for $X_t$, that is, $\mathbb{P}(X_t \in A) = \mu(A)$ holds for all $A\in\Sigma_{\XX}$ and $t \geq 0$, provided $X_0$ is distributed according to~$\mu$.
Under Assumption~\ref{as:dyn} (b), Definition \ref{d:koopman} can be extended to the Lebesgue spaces $L^p_\mu(\XX)$, $1 \leq p < \infty$, i.e.\ the Banach spaces of all (equivalence classes of) measurable functions $f:\mathbb{X}\to\R$ with $\int_\mathbb{X}|f|^p \,\text{d}\mu < \infty$. Then, the Koopman operators $\mathcal{K}^t$ form a strongly continuous semigroup of contractions on all spaces $L^p_\mu(\XX)$, see \cite{BAKRY2013}. The functions in any of these spaces are often referred to as \textit{observables}.

%Since solutions to \eqref{e:SDE} are continuous, the semigroups above are strongly continuous~\cite{Oksendal2013}, which %naturally prompts the following definition.

Next, we recall the definition of the generator associated to the semigroup $\mathcal{K}_t$:
\begin{definition}[Koopman generator]
\label{d:koopman_generator}
The infinitesimal generator $\mathcal{L}$ is defined via
\begin{align}
\label{e:def_generator}
\mathcal{L}f := \lim_{t\to 0} \frac{(\mathcal{K}^t - \operatorname{Id})f}{t}
\end{align}
for all $f \in D(\mathcal{L})$, where $D(\mathcal{L})$ is the set of functions for which the limit~\eqref{e:def_generator} exists in the appropriate topology.
\end{definition}

\noindent For sufficiently smooth functions $f$, Ito's Lemma \cite{Oksendal2013} shows that the generator acts as a second order differential operator, defined in terms of the coefficients of \eqref{e:SDE}, i.e.
\begin{align}\label{e:repL}
\mathcal{L} = F \cdot \nabla  + \tfrac12 \sigma \sigma^\top : \nabla^2
\end{align}
with $A: B := \sum_{i,j}a_{i,j}b_{i,j}$ being the standard Frobenius inner product for matrices. In what follows, we will focus exclusively on the Koopman semigroup on the Hilbert space $L^2_\mu(\XX)$ with inner product $\langle f, g \rangle_\mu = \int_{\XX} f g \, \mathrm{d}\mu$. As the semigroup is strongly continuous on $L^2_\mu(\XX)$ by our assumptions, by standard semigroup theory, the domain $D(\mathcal{L})$ together with the graph norm forms a dense Banach space in $L^2_\mu(\XX)$.

%A central assumption in the remainder of this work will be the following.
% \MS{Vieleicht die Annahme hinten hin packen wo sie gebraucht wird.}
%\begin{assumption}[Exponential stability]
%	\label{assumption:exp_stability}
%
%\end{assumption}

\subsection{Extended Dynamic Mode Decomposition}
\label{subsec:edmd0}
\noindent
In this part we introduce the data-driven finite-dimensional approximation by eDMD of the Koopman generator defined in~\eqref{e:def_generator}, see, e.g., \cite{WILLIAMS2015,KLUS2016b,KLUS2018b}. To this end, for a fixed set of linearly independent observables $\psi_1,\ldots,\psi_N\in D(\calL)$, we consider the finite-dimensional subspace 
\begin{align*}
\mathbb{V} := \operatorname{span}\{\{\psi_j\}_{j=1}^N\} \subset D(\mathcal{L}).
\end{align*}
Let $P_\mathbb{V}$ denote the orthogonal projection onto $\mathbb{V}$. We define the Galerkin projection of the Koopman generator by %
%\begin{align*}
	$\LV: ={P}_{\mathbb{V}} \mathcal{L}\restrict{\mathbb{V}}$. %
%\end{align*}
Note that this is not the restriction of $\mathcal{L}$ onto $\mathbb{V}$, as the image is also projected back onto $\mathbb{V}$. If $\mathbb{V}$ is an invariant set under the action of the generator, then $\LV= \mathcal{L}\restrict{\mathbb{V}}$ holds. %
\MS{As $\dim\mathbb{V} = N$, the linear operator~$\LV : \mathbb V\to\mathbb V$ may be represented by a matrix. In what follows, we denote the matrix representation of $\calL_{\mathbb V}$ in terms of the basis functions $\psi_1,\ldots,\psi_N$ by the same symbol $\calL_{\mathbb V}$ as the operator itself in a slight abuse of notation}. Thus, using~\cite{Klus2020}, we get %
\begin{align*}
	\LV = C^{-1}A
\end{align*}
with $C, A \in \mathbb{R}^{N\times N}$ defined by $C_{i,j}=\langle\psi_i,\psi_j\rangle_{\Lmu}$ and $A_{i,j} =\langle\psi_i,\mathcal{L}\psi_j\rangle_{\Lmu}$. %
The norm of the isomorphism from $\mathbb{V}$ to $\mathbb{R}^N$ depends on the smallest resp.\ largest eigenvalues of $C$, cf.\ Proposition~\ref{p:normeq} in Appendix~\ref{s:isomorphism}.

Consider data points $x_0, \ldots, x_{m-1} \in \mathbb{X}$. \MS{In the following, these data is either drawn from a trajectory of an ergodic system or sampled independent and identically distributed (i.i.d.). We state this as the following assumption, using the notation:
	$$
	L^2_{\mu,0}(\XX) := \{f \in L^2_\mu(\XX):\, \langle f, 1 \rangle_\mu = 0\}.
	$$
\begin{assumption}
	\label{as:data}
Let Assumption~\ref{as:dyn} hold and assume either of the following.
\begin{itemize}
	%\item[(iid)] The data is drawn i.i.d.\ w.r.t\ an invariant measure with positive densitiy in the sense of Definition~\ref{def:invariant_measure} or drawn i.i.d.\ w.r.t.\ the Lebesgue measure and contained in the compact set $\mathbb{X}$. 
	\item[(iid)] The data is drawn i.i.d.\ from the measure specified via Assumption~\ref{as:dyn}.
	%In the latter, further assume that $\mathbb{X}$ is compact, the data is drawn i.i.d.\ from the Lebesgue measure and contained in $\mathbb{X}$.
%The probability measure of the SDE is invariant with positive density and the data is drawn i.i.d.\ from this measure.
	\item[(erg)] Assumption~\ref{as:dyn}.(b) holds and the data is obtained as snapshots from a single ergodic trajectoy, that is, from a single long trajectory of the dynamics \eqref{e:SDE} with $x_0$ drawn from the unique invariant measure $\mu$. Further assume the Koopman semigroup is exponentially stable on $L^2_{\mu,0}(\XX)$, i.e. $\|\mathcal{K}^t \|_{L^2_{\mu, 0}(\XX)} \leq Me^{-\omega t}$ for some $M\geq 1$, $\omega > 0$. 
\end{itemize}
\end{assumption}}

 %which can either be drawn i.i.d. from $\mu$, or can be snapshots $x_k = X_{k\Delta_t}$ with stepsize $\Delta_t>0$ from a single long trajectory of the dynamics \eqref{e:SDE} with $x_0$ drawn from the unique invariant measure $\mu$ (called an \textit{ergodic} trajectory in what follows).
  %As a consequence of Assumption \ref{assumption:invariant_measure}, also in this case the $x_k$ are $\mu$-distributed.

\noindent %
Let us form the transformed data matrices
\begin{align*}
\Psi(X) &:= \left( \left. \left(\begin{smallmatrix}
\psi_1(x_0)\\
:\\
\psi_N(x_0)
\end{smallmatrix}\right)\right| \ldots \left| \left(\begin{smallmatrix}
\psi_1(x_{m-1})\\
:\\
\psi_N(x_{m-1})
\end{smallmatrix}\right)\right. \right)\\
\mathcal{L}\Psi(X) &:= \left( \left. \left(\begin{smallmatrix}
(\mathcal{L}\psi_1)(x_0)\\
:\\
(\mathcal{L}\psi_N)(x_0)
\end{smallmatrix}\right)\right| \ldots \left| \left(\begin{smallmatrix}
(\mathcal{L}\psi_1)(x_{m-1})\\
:\\
(\mathcal{L}\psi_N)(x_{m-1})
\end{smallmatrix}\right)\right. \right).
\end{align*}
The evaluation of $\mathcal{L}$ can be realized via the representation~\eqref{e:repL}.
The empirical estimator for the Galerkin projection $\LV$ is then given by 
\begin{align*}
\tL = \tC^{-1} \tA
\end{align*}
with $\tC = \tfrac{1}{m} \Psi(X) \Psi(X)^\top$, $\tA = \tfrac{1}{m} \Psi(X) \mathcal{L}\Psi(X)^\top \in \mathbb{R}^{N\times N}$. 
%\begin{align*}
%\tC = \tfrac{1}{m} \Psi(X) \Psi(X)^\top \qquad\text{and}\qquad \tA = \tfrac{1}{m} \Psi(X) \mathcal{L}\Psi(X)^\top.
%\end{align*}
In all scenarios of \MS{of Assumption~\ref{as:data}}, we have with probability one that
\begin{enumerate}
\item [(1)] $\tL$ is well-defined for large enough $m$, that is, $\tC$ is invertible, and
\item [(2)] $\tL$ converges to $\LV$ for $m\rightarrow \infty$, see, e.g., \cite{KLUS2018b,Klus2020}.
\end{enumerate}
For the case of a long trajectory, this result follows from ergodic theory, which is concerned with the convergence of time averages to spatial averages as the data size grows to infinity~\cite{Beck1957}. Ergodic theory particularly applies to systems with a unique invariant measure.

%\begin{align*}
%\Psi(X) &:= \left( \left. \left(\begin{smallmatrix}
%\psi_1(x_1)\\
%:\\
%\psi_N(x_1)
%\end{smallmatrix}\right)\right| \ldots \left| \left(\begin{smallmatrix}
%\psi_1(x_m)\\
%:\\
%\psi_N(x_m)
%\end{smallmatrix}\right)\right. \right)\\
%\Psi(Y) &:= \left( \left. \left(\begin{smallmatrix}
%\nabla \psi_1({x}_1)\cdot \dot{x}_1)\\
%:\\
%\nabla \psi_N(x_1)\cdot \dot{x}_1
%\end{smallmatrix}\right)\right| \ldots \left| \left(\begin{smallmatrix}
%\nabla \psi_1(x_m)\cdot \dot{x}_m\\
%:\\
%\nabla \psi_N(x_m)\cdot \dot{x}_m
%\end{smallmatrix}\right)\right. \right)
%\end{align*}
%\FN{I suggest we drop Remark 4 entirely if we are not treating this setting in more depth.}
%\begin{remark}
%	If we additionally have as basis $V$ of orthonormal eigenvectors of $C(\Psi)$, we have that $C = V\Sigma^2 V^\top$. Further we choose $U_r = V_{:,1:r}\Sigma^{-1}_{1:r,1:r}$, set the transformed basis to $\eta_r =  U_r^\top \psi $ and have that the generator in this basis is given by
%	\begin{align*}
%	\LV(\eta_r) = (U^\top_r C U_r)^{-1} U_r^\top AU_r = A(\eta_r),
%	\end{align*}
%	where $(A(\eta_r))_{i,j} = \langle (\eta_r)_i, \mathcal{L}(\eta_r)_j\rangle_{\Lmu}$. The same obviously can be done for the empirical estimator $\tL$.
%\end{remark}

\MS{\subsection{Error bounds on approximations of projected Koopman generator and operator}\label{ss:projgen}
\noindent
Next, we quantify the approximation quality of the data-driven finite-dimensional approximation of the Koopman generator, i.e., for a given linear space $\mathbb V$ of observables and data $x_0,\ldots,x_{m-1}\in\mathbb{X}$, we aim to estimate 
\begin{align*}
\| \LV - \tL\|_F = \|C^{-1}A - \tilde{C}_m^{-1}\tilde{A}_m\|_F.
\end{align*}

\subsubsection{Concentration bounds for random matrices}
%\noindent\textbf{Concentration Bounds for Scalar Variables.}
We start by deriving entry-wise error bounds for the data-driven mass and stiffness matrix, respectively. Since most of the arguments are significantly simpler for i.i.d.~sampling, cf.\ Remark~\ref{rem:iid} at the end of this subsection, we first consider the more involved situation, i.e.\ ergodic sampling. This is of particular interest as simulation data of the dynamics~\eqref{e:SDE} can, then, be directly used. 
%As we are particularly interested in using simulation data of the dynamics~\eqref{e:SDE} for Koopman modeling, we treat the case of ergodic sampling under Assumption~\ref{as:data}.(erg) in what follows. The i.i.d.\ setting of Assumption~\ref{as:data}.(iid) significantly simplifies most arguments and hence will be discussed at the end of this section in Remark~\ref{rem:iid}.%\\

%\noindent 
For $x \in \mathbb{X}$, consider a centered scalar random variable
\[ \phi: \, \mathbb{X} \mapsto \mathbb{R}, \quad \int_\mathbb{X} \phi(x) \,\mathrm{d}\mu(x) = 0. \]
We denote its variance with respect to the invariant measure by
\[ \sigma^2_\phi = \mathbb{E}^\mu[\phi^2] = \|\phi\|^2_{L^2_\mu}.\]
Moreover, we set $\phi_k = \phi(x_k)$ for given data points~$x_k$, $k \in \{0,1,\ldots,m-1\}$, and define the averaged random variable
\begin{align*}
\bar{\phi}_m &:= \frac{1}{m}\sum_{k= 0 }^{m-1} \phi_k.
\end{align*}
In Lemma~\ref{lemma:co-variances} below, we quantify the variance of the averaged random variable $\bar{\phi}_m$. The key point is the decomposition of the variance into an asymptotic contribution, independent of $m$, and a second contribution, which decays with an explicitly given (polynomial) dependence on the amount of data~$m$.
\begin{lemma}
\label{lemma:co-variances}
Let Assumption~\ref{as:data}.(erg) hold. Then we have
\begin{align}
\label{e:remainder}
\sigma^2_{\bar{\phi}_m} &= \frac{1}{m}\left[ \sigma_{\phi, \infty}^2 - R_\phi^m\right].
\end{align}
The asymptotic variance $\sigma_{\phi, \infty}^2$ and the remainder term $R_\phi^m$ are given by
\begin{align*}
\sigma_{\phi, \infty}^2 &= \sigma_{\phi}^2 + 2 \sum_{l=1}^\infty \langle \phi, \, \mathcal{K}^{l\Delta_t} \phi \rangle_\mu,
& R_\phi^m &= 2 \sum_{l=m}^\infty \langle \phi, \, \mathcal{K}^{l\Delta_t} \phi \rangle_\mu + \frac{2}{m} \sum_{l=1}^{m-1} l \langle \phi, \, \mathcal{K}^{l\Delta_t} \phi \rangle_\mu.
\end{align*}
\end{lemma}
\begin{proof}
We repeat the proof given in \cite[Section 3.1.2]{Lelievre:2016aa} for the sake of illustration:
\begin{align}
\nonumber \sigma^2_{\bar{\phi}_m} &= \frac{1}{m^2} \sum_{k, l = 0}^{m-1} \mathbb{E}^\mu\left[ \phi_k \, \phi_l \right]  
= \frac{1}{m} \sigma^2_\phi+ \frac{2}{m^2} \sum_{k=0}^{m-1} \sum_{l=k+1}^{m-1} \mathbb{E}^\mu\left[ \phi_k\, \phi_l \right] \\
\nonumber &= \frac{1}{m} \left[\sigma^2_\phi + \frac{2}{m} \sum_{k=0}^{m-1} \sum_{l=k+1}^{m-1} \mathbb{E}^\mu\left[ \phi_0\, \phi_{l-k} \right] \right]
= \frac{1}{m} \left[\sigma^2_\phi + \frac{2}{m} \sum_{k=0}^{m-1}\sum_{l=1}^{m-k-1} \mathbb{E}^\mu \left[ \phi_0  \, \phi_{l} \right] \right] \\
\nonumber &= \frac{1}{m} \left[\sigma^2_\phi + \frac{2}{m}\sum_{l=1}^{m-1} (m -l)  \mathbb{E}^\mu \left[ \phi_0  \, \phi_{l} \right] \right]
= \frac{1}{m} \left[\sigma^2_\phi + 2\sum_{l=1}^{m-1} (1 - \tfrac{l}{m}) \langle\phi,\, \mathcal{K}^{l\Delta_t} \phi \rangle_\mu \right].
\end{align}
The result follows by adding and subtracting the term $2 \sum_{l=m}^\infty \langle \phi, \, \mathcal{K}^{l\Delta_t} \phi \rangle_\mu$.
\end{proof}
\begin{remark}
The assumption of exponential stability is satisfied, for example, if the generator $\mathcal{L}$ is self-adjoint \braces{also known as detailed balance or reversibility} and additionally
%the symmetric part of $\mathcal{L}$
satisfies a Poincar\'e or spectral gap inequality {\textup{\cite{Lelievre:2016aa}}}. The requirement $\langle f, 1 \rangle_\mu = 0$ is necessary, as the constant function is invariant for $\mathcal{K}^t$.
\end{remark}

\begin{remark}
The proof of Lemma~\rmref{lemma:co-variances} shows that $\sigma^2_{\phi, \infty} = \lim_{m \rightarrow \infty} \sigma^2_{\bar{\phi}_m} \geq 0$, hence it can indeed by interpreted as a variance. 

For reversible systems, we have $\langle \phi,\, \mathcal{K}^{l \Delta_t} \phi \rangle_\mu \geq 0$ by symmetry of the Koopman operator. Therefore, $\sigma^2_{\phi, \infty} \geq \sigma^2_\phi > 0$ is guaranteed in this case, and the variance $\sigma^2_{\bar{\phi}_m}$ approaches $\frac{1}{m}\sigma^2_{\phi, \infty}$ from below.
\end{remark}

\noindent
Next, we derive an estimate for the remainder term in terms of the number~$m$ of data points.

\begin{lemma}
\label{lemma:bound_remainder}
Let Assumption~\ref{as:data}.(erg) hold, and set $q = e^{-\omega\Delta_t} < 1$. Then
%	\[ |R_\phi^m | \leq \frac{2\sigma^2_{\phi}}{m}\frac{q(1 - q^m)}{(1-q)^2}. \]
\[ |R_\phi^m | \leq \frac{2\sigma^2_{\phi}}{m}\frac{q}{(1-q)^2}. \]
\end{lemma}
\begin{proof}
We first observe that by the Cauchy Schwarz inequality
\begin{align*}
|\langle \phi, \mathcal{K}^{l\Delta_t} \phi \rangle_\mu |
&\leq \|\mathcal{K}^{l \Delta_t} \|_{L(L^2_{\mu, 0}(\mathbb{X}), L^2_{\mu}(\mathbb{X}))} \|\phi\|^2_{L^2_{\mu}(\mathbb{X})} \leq e^{-\omega\Delta_t l} \sigma^2_\phi,
\end{align*}
and therefore:
\begin{align*}
|R_\phi^m | &\leq 2\sigma^2_{\phi} \left[ \sum_{l=m}^\infty e^{-\omega\Delta_t l} + \frac{1}{m} \sum_{l=1}^{m-1} l e^{-\omega\Delta_t l} \right] \\
&= 2\sigma^2_{\phi} \left[ \frac{q^m}{1-q} + \frac{1}{m} \frac{(m-1)q^{m+1} - mq^m + q }{(1-q)^2} \right] \\
&= \frac{2\sigma^2_{\phi}}{m}\frac{q(1 - q^m)}{(1-q)^2} \leq \frac{2\sigma^2_{\phi}}{m}\frac{q}{(1-q)^2} .
\end{align*}
In the second line, we have used the geometric series for the first term, and a similar identity for the sum $\sum_{l=1}^\infty l q^l, \,q < 1$. The third line is obtained by direct simplification.
\end{proof}

\noindent We can now combine the results of Lemmas~\ref{lemma:co-variances} and~\ref{lemma:bound_remainder} in order to obtain a concentration bound for a centered, matrix-valued random variable. To this end, we consider an $N\times N$ random matrix $\Phi$ with all entries $\phi_{ij} \in L^2_{\mu, 0}$ centered. We define $\Phi_k$ and $\overline{\Phi}_m$ as for the scalar case, i.e., $\Phi_k = \Phi(x_k)$ and $\overline{\Phi}_m = \tfrac{1}{m}\sum_{k=0}^{m-1}\Phi_k$. 

\begin{proposition}\label{thm:prob_error_bounds_xi}
Let Assumption~\ref{as:data}.(erg) hold,, set $q = e^{-\omega\Delta_t}$, and assume $\sigma^2_{\phi_{ij}, \infty} > 0$ for all $(i, j)$. Let $\Phi \in \mathbb{R}^{N \times N}$ be a centered, matrix-valued random variable in $L^2_\mu$. Denote the matrices of all entry-wise variances and asymptotic variances by
\begin{align*}
\Sigma_{\Phi} &= \left(\sigma_{\phi_{ij}} \right)_{i,j=1}^N, & \Sigma_{\Phi, \infty} &= \left(\sigma_{\phi_{ij}, \infty} \right)_{i,j=1}^N
\end{align*}
 Then, for any given $\delta > 0$, and $m \in \mathbb{N}$, we have with probability at least $1 - \delta$ that
\begin{align}\label{eq:ineq_m}
\|\overline{\Phi}_m\|_F \leq \frac{N}{\sqrt{m \delta}} \left[ \|\Sigma_{\Phi,\infty}\|^2_F + \frac{2q}{m (1- q)^2} \|\Sigma_{\Phi} \|^2_F \right]^{1/2}.
\end{align}
For reversible systems, we obtain the simplified bound
\begin{align}
\label{eq:ineq_m_rev}
\|\overline{\Phi}_m\|_F \leq \frac{N}{\sqrt{m \delta}} \|\Sigma_{\Phi,\infty}\|_F.
\end{align}
\end{proposition}
\begin{proof}
Noting that $[\overline{\Phi}_m]_{ij} = [\ol{\phi_{ij}}]_m$, the scalar Chebyshev inequality and the result of Lemma~\ref{lemma:co-variances}, yield for all $(i, j)$ :
\begin{align*}
\mathbb{P}\left([\overline{\Phi}_m]_{ij}^2 \leq \varepsilon^2 \right)
&\geq 1 - \frac{\sigma^2_{[\ol{\phi_{ij}}]_m}}{\varepsilon^2} = 1 - \frac{\frac{1}{m}[\sigma^2_{\phi_{ij},\infty} - R_{\phi_{ij}}^m]}{\varepsilon^2}\\
&\geq 1 - \frac{1}{m \varepsilon^2}\left[\sigma^2_{\phi_{ij},\infty} + \frac{2 \sigma^2_{\phi_{ij}} q}{m (1- q)^2}\right].
\end{align*}
The second term on the right hand side does not exceed $\frac{\delta}{N^2}$ if 
\begin{equation*}
\veps^2 \geq \frac{N^2}{m \delta}\left[\sigma^2_{\phi_{ij},\infty} + \frac{2 \sigma^2_{\phi_{ij}} q}{m (1- q)^2}\right],
\end{equation*}
in other words, there is a set of trajectories of probability at least $1 - \frac{\delta}{N^2}$ such that
\begin{equation*}
[\overline{\Phi}_m]_{ij}^2 \leq \frac{N^2}{m \delta}\left[\sigma^2_{\phi_{ij},\infty} + \frac{2 \sigma^2_{\phi_{ij}} q}{m (1- q)^2}\right].
\end{equation*}
On the intersection of these sets, we have that
\begin{align*}
\| \overline{\Phi}_m\|_F \leq \frac{N}{\sqrt{m \delta}} \left[ \|\Sigma_{\Phi,\infty}\|^2_F + \frac{2q}{m (1- q)^2} \|\Sigma_{\Phi} \|^2_F \right]^{1/2},
\end{align*}
and the probability of the intersection is at least $1 - \delta$ by Lemma~\ref{lem:probabilities}. In the reversible case, we know that $R_{\phi_{ij}}^m \geq 0$ for all $(i, j)$, and therefore
\begin{align}
\label{eq:error_phi_ij_rev}
\mathbb{P}\left([\overline{\Phi}_m]_{ij}^2 \leq \varepsilon^2 \right) &\geq 1 - \frac{1}{m \varepsilon^2}\sigma^2_{\phi_{ij},\infty}.
\end{align}
The simplified bound~\eqref{eq:ineq_m_rev} follows by repeating the above argument starting from this inequality.
\end{proof}
\begin{remark}[I.i.d.\ sampling]
	\label{rem:iid}
If the data are sampled i.i.d., that is, Assumption~\ref{as:data}.(iid) hold instead of Assumption~\ref{as:data}.(erg), then by standard results, one has $\sigma^2_{\bar{\phi}_m} = \frac{1}{m}\sigma^2_\phi$. The bounds from Proposition~\ref{thm:prob_error_bounds_xi} simplify  significantly in this case. By the Chebyshev inequality:
\begin{align*}
\mathbb{P}\left([\overline{\Phi}_m]_{ij}^2 \leq \varepsilon^2 \right)
&\geq 1 - \frac{\frac{1}{m} \sigma^2_{\phi_{ij}}}{\varepsilon^2},
\end{align*}
which leads to the following error estimate for fixed $m \in \mathbb{N}$ and $\delta > 0$:
\begin{align}
\label{eq:ineq_m_iid}
\|\overline{\Phi}_m\|_F \leq \frac{N}{\sqrt{m \delta}} \|\Sigma_{\Phi} \|_F.
\end{align}
The setting of sampling via the Lebesgue measure on a compact set $\mathbb{X}$ was thoroughly considered in \cite{ZhanZuaz21}.
\end{remark}
\subsubsection{Error bound for the projected generator}
%\noindent\textbf{Error Bound for the Projected Generator}:
Next, we deduce our first main result by applying the probabilistic bounds obtained in Proposition~\ref{thm:prob_error_bounds_xi} to estimate the error for the data-driven Galerkin projection $\tL$.
\begin{theorem}[Approximation error: probabilistic bound]
\label{t:generatorestimate}
Let Assumption~\ref{as:data} hold. Then, for any error bound $\tilde\varepsilon > 0$ and probabilistic tolerance~$\tilde\delta \in (0,1)$, we have %the estimate
\begin{align}\label{e:prob_est}
\mathbb{P}\left(\|\LV- \tL\|_F\leq \tilde\varepsilon\right) \geq 1-\tilde\delta
\end{align}
for any amount~$m \in \mathbb{N}$ of data points such that the following hold with $$\varepsilon = \min\left\{1,\frac{1}{\|A\|\|C^{-1}\|}\right\}\cdot\frac{\|A\|\tilde\veps}{2\|A\|\|C^{-1}\| + \tilde\veps}\quad \text{and}\quad \delta = \frac{\tilde\delta}3.$$
\begin{itemize}
\item In case of ergodic sampling, i.e., Assumption~\ref{as:data}.(erg),
\begin{align*}
m \geq  \frac{N^2}{\delta\varepsilon^2} \left[ \|\Sigma_{\Phi,\infty}\|^2_F + \frac{2q}{m (1- q)^2} \|\Sigma_{\Phi} \|^2_F \right]
\end{align*}
\item In case of ergodic sampling, i.e., Assumption~\ref{as:data}.(erg), of a reversible system
\begin{align*}
m \geq  \frac{N^2}{\delta\varepsilon^2} \|\Sigma_{\Phi,\infty}\|^2_F.
\end{align*}
\item In case of i.i.d.\ sampling, i.e., Assumption~\ref{as:data}.(iid),
\begin{align*}
m \geq \frac{N^2}{\delta\varepsilon^2} \|\Sigma_{\Phi} \|^2_F.
\end{align*}
\end{itemize}
%
%\begin{align}
%\label{eq:ineq_m_rev}
%\|\overline{\Phi}_m\|_F \leq \frac{N}{\sqrt{m \delta}} \|\Sigma_{\Phi,\infty}\|_F.
%\end{align}
%\begin{align}
%\label{eq:ineq_m_iid}
%\end{align}
%$(\varepsilon,\delta)$ given by $\left( \min\left\{ \tfrac{1}{2\|C^{-1}\|},\tfrac{\varepsilon}{2\|A\|\|C^{-1}\|^2},\tfrac{\varepsilon}{2\left(\|C^{-1}\|+\tfrac{\varepsilon}{2\|A\|}\right)}  \right\},\tfrac{\delta}{2}\right)$
%%%%%
%	$$
%	\varepsilon \rightarrow \min\{ 2\|A\||C^{-1}\| + \varepsilon \|C^{-1}\|, 
%	$$
%	\begin{align}\label{e:mbound}
%		m \geq c\frac {N^2}{\varepsilon \delta} \left(\|C^{-1}\| + \tfrac{\varepsilon}{2\|A\|}\right)
%	\end{align}
%	with constant~$c$ given by
%	\begin{align*}
%		c = 4 \max\{\lambda_\xi^{\max},\lambda_\chi^{\max}\}(1 + \eta)\max\{\|C^{-1}\| \|A\| ,1\}.
%	\end{align*}
%	\begin{align}
%	\label{e:mbound1}
%m \geq \frac{4N^2 \max\{\lambda_\xi^{\max},\lambda_\chi^{\max}\}(1 + \eta){\left(\|C^{-1}\| + \tfrac{\varepsilon}{2\|A\|}\right) \|C^{-1}\|}}{ \tfrac{\varepsilon}{\|A\|}\delta}
%	\end{align}
%		and
%	\begin{align}
%	\label{e:mbound2}
%	m \geq \frac{4N^2 \max\{\lambda_\xi^{\max},\lambda_\chi^{\max}\}(1 + \eta)\left(\|C^{-1}\|+\tfrac{\varepsilon}{2\|A\|}\right)}{\varepsilon \delta} 
%	\end{align}	
\end{theorem}
\begin{proof}
In this proof, we will omit the subscript for the norm and set $\|\cdot\| = \|\cdot\|_F$. Let us introduce the centered matrix-valued random variables
\begin{align*}
\Phi_C(x) := \Psi(x)\Psi(x)^\top - C
\qquad\text{and}\qquad
\Phi_A(x) := \Psi(x)\calL\Psi(x)^\top - A,
\end{align*}
where $\Psi = [\psi_1,\ldots,\psi_N]^\top$. Then $\wt C_m - C = [\ol{\Phi_C}]_m$ and $\wt A_m - A = [\ol{\Phi_A}]_m$. Hence, we may apply Proposition~\ref{thm:prob_error_bounds_xi} to these matrix-valued random variables. 
%First, by the choice of $m$ and Proposition~\ref{thm:prob_error_bounds_xi}, we have
%\begin{align*}
%\mathbb{P}\left(\|\tC-C\|\leq \tfrac{\tfrac{\varepsilon}{2\|A\|}}{\left(\|C^{-1}\| + \tfrac{\varepsilon}{2\|A\|}\right) \|C^{-1}\|}\right) &\geq 1-\tfrac{\delta}{2} \\
%\mathbb{P}\left(\|\tC-C\|\leq \tfrac{\varepsilon}{2\left(\|C^{-1}\| + \tfrac{\varepsilon}{2\|A\|}\right) \|C^{-1}\|\|A\|}\right) &\geq 1-\tfrac{\delta}{2} \\
%\mathbb{P}\left(\|\tA-A\|\leq \tfrac{\varepsilon}{2\left(\|C^{-1}\|+\tfrac{\varepsilon}{2\|A\|}\right)}\right) &\geq 1-\tfrac{\delta}{2},
%\end{align*}
First, by the choice of $m$ above we have
\begin{align*}
%\mathbb{P}\left( \|C-\wt C_m\|\le \frac{\tilde\veps}{2\|A\|\|C^{-1}\|^2 + \|C^{-1}\|\tilde\veps}\right) &\geq 1-\tfrac{\tilde\delta}{3} \\
%\mathbb{P}\left(\|\tA-A\|\leq \tfrac{\tilde\varepsilon}{2\left(\|C^{-1}\|+\tfrac{\tilde\varepsilon}{2\|A\|}\right)}\right) &\geq 1-\tfrac{\tilde\delta}{3},
\mathbb{P}\left( \|C-\wt C_m\|\le \frac{R}{\|A\|\|C^{-1}\|}\right) \geq 1-\tfrac{\tilde\delta}{3}
\qquad\text{and}\qquad
\mathbb{P}\left(\|\tA-A\|\leq R\right) \geq 1-\tfrac{\tilde\delta}{3},
\end{align*}
where
$$
R := \frac{\|A\|\tilde\veps}{2\|A\|\|C^{-1}\| + \tilde\veps} = \frac{\tilde\veps}{2\left(\|C^{-1}\| + \frac{\tilde\veps}{2\|A\|}\right)}.
$$
Moreover, we compute
\begin{align*}
\|\tC^{-1}-C^{-1}\|
&= \|\tC^{-1}(C - \tC)C^{-1}\|\leq \|C^{-1}\|\|C-\tC\| \left(\|\tC^{-1}- C^{-1}\| + \|C^{-1}\|\right)
\end{align*}
which implies
$$
\|\tC^{-1}-C^{-1}\| \leq \frac{\|C^{-1}\|^2\|C-\tilde{C}_m\|}{1-\|C^{-1}\|\|C-\tilde{C}_m\|}.
$$ 
Hence, by straightforward computations we obtain
\begin{align*}
\mathbb{P}\left(\|\tC^{-1}-C^{-1}\| \leq \tfrac{\tilde\veps}{2\|A\|}\right)
&\geq \mathbb{P}\left(\frac{\|C^{-1}\|^2\|C-\tilde{C}_m\|}{1-\|C^{-1}\|\|C-\tilde{C}_m\|} \leq \tfrac{\tilde\veps}{2\|A\|}\right)\\
&=\mathbb{P}\left(\|C-\wt C_m\|\le \frac{R}{\|A\|\|C^{-1}\|}\right)  \geq 1-\tfrac{\tilde\delta}{3}.
\end{align*}
and
\begin{align*}
\mathbb{P}\left(\|\tA-A\|\leq \tfrac{\tilde\veps}{2\|\tC^{-1}\|}\right)
&\geq \mathbb{P}\left(\|\tA-A\|\leq \tfrac{\tilde\veps}{2\left(\|C^{-1}\|+ \|\tC^{-1}-C^{-1}\| \right)}\right)\\
&\geq \mathbb{P}\left(\|\tA-A\|\leq \tfrac{\tilde\veps}{2\left(\|C^{-1}\|+\tfrac{\tilde\veps}{2\|A\|}\right)} \;\wedge\; \|\tC^{-1}-C^{-1}\| \leq \tfrac{\tilde\veps}{2\|A\|}  \right) \\
&\geq \left(1-\tfrac{\tilde\delta}{3}\right) + \left(1-\tfrac{\tilde\delta}{3}\right) - 1 = 1-\tfrac{2\tilde\delta}{3}.
\end{align*}
Thus, we conclude
\begin{align*}
\mathbb{P}(\|C^{-1}A -\tC^{-1}\tA\|\leq \tilde\veps)
&= \mathbb{P}\left(\|\tC^{-1}(A-\tA) + \left(C^{-1}-\tC^{-1}\right) A\|\leq \tilde\veps\right) \\
&\geq \mathbb{P}\left(\|\tC^{-1}\|\|A-\tA\| + \|C^{-1}-\tC^{-1}\|\|A\|\leq \tilde\veps\right)\\
&\geq \mathbb{P}\left(\|A-\tA\| \leq \tfrac{\tilde\veps}{2\|\tC^{-1}\|} \wedge \|C^{-1}-\tC^{-1}\|\leq \tfrac{\tilde\veps}{2\|A\|}\right)\\
%&=  \mathbb{P}\left(\|A-\tA\| \leq \tfrac{\tilde\veps}{2\|\tC^{-1}\|}\right) + \mathbb{P}\left(\|C^{-1}-\tC^{-1}\|\leq \tfrac{\tilde\veps}{2\|A\|}\right)\\
%&\qquad  -\mathbb{P}\left(\|A-\tA\| \leq \tfrac{\tilde\veps}{2\|\tC^{-1}\|} \vee \|C^{-1}-\tC^{-1}\|\leq \tfrac{\tilde\veps}{2\|A\|}\right)\\
&\geq (1-\tfrac{2\tilde\delta}{3}) + (1-\tfrac{\tilde\delta}{3}) - 1 = 1-\tilde\delta,
\end{align*}
which is \eqref{e:prob_est}.
\end{proof}
\noindent A similar result as Theorem~\ref{t:generatorestimate} was obtained for ODE systems in \cite{ZhanZuaz21} under the assumption that the data is drawn i.i.d.

An immediate consequence of the estimate on the generator approximation error is a bound on the error of the trajectories. To this end, consider the systems
\begin{align}
\label{e:z}
&&\dot{z} &= \LV z  &&z(0)=z_0\\
\label{e:tz}
&&\dot{\tilde{z}} &= \tL \tilde z  &&\tilde{z}(0)={z}_0.
\end{align}
where $z_0\in \mathbb{R}^n$, which represents an ODE in terms of the coefficients in the basis representation of elements of $\mathbb{V}$. We will leverage the error bound obtained in Theorem \ref{t:generatorestimate} to derive an estimate on the resulting prediction error in the observables, i.e., $\|z(t)-\tilde{z}(t)\|_2$. Note that in view of the isomorphism $\mathbb{V}\simeq \mathbb{R}^N$ this also directly translates to an error estimate for trajectories in $\mathbb{V}$.
\begin{corollary}
\label{c:trajest}
Let Assumption~\ref{as:data} hold. Then for any $T>0$ and $\delta,\varepsilon>0$ there is $m_0\in \mathbb{N}$ such that for $m\geq m_0$ data points we have
\begin{align*}
\min_{t \in [0,T]} \mathbb{P}\big(\|z(t) - \tilde z(t)\|_2 \leq \varepsilon\big) \geq 1-\delta.
\end{align*}
\end{corollary}
\begin{proof}
See Appendix~\ref{s:errorbound_traj}. 
\end{proof}
A sufficient amount of data~$m_0$ can be easily specified by combining the calculations displayed in the proof of Corollary~\ref{c:trajest}, i.e.\ Gronwall's inequality and Condition \eqref{eq:ineq_m}. Under additional assumptions on the Koopman semigroup generated by $\LV$, e.g., stability, one can refine this estimate or render it uniform in $T$, cf.\ Corollary~\ref{c:refinements} in Appendix~\ref{s:errorbound_traj}.
\subsection{Error bound for the projected Koopman operator}
\noindent
Similar to the derivation of the probabilistic bound on the projected generator, a bound on the Koopman operator is possible. We briefly sketch the main steps of the argumentation. Let $t = l\Delta_t$ for some $l\in\N$ and again choose a subspace $\mathbb{V} = \operatorname{span}\{\{\psi_j\}_{j=1}^N\}\subset L^2_\mu(\mathbb{X})$ (which, in contrast to the generator-based setting, is not required to be contained in the domain). The restricted Koopman operator on this subspace is defined via
\begin{align*}
\mathcal{K}^{t}_\mathbb{V}:= P_\mathbb{V}\mathcal{K}^{t}\restrict{\mathbb{V}} = C^{-1} A,
\end{align*}
where
$$
C_{i,j}=\langle\psi_i,\psi_j\rangle_{\Lmu}
\qquad\text{and}\qquad
A_{i,j} =\langle\psi_i,\mathcal{K}^{t}\psi_j\rangle_{\Lmu}.
$$
Define the data matrices 
\begin{align*}
\Psi(X) &:= \left( \left. \left(\begin{smallmatrix}
\psi_1(x_0)\\
:\\
\psi_N(x_0)
\end{smallmatrix}\right)\right| \ldots \left| \left(\begin{smallmatrix}
\psi_1(x_{m-l-1})\\
:\\
\psi_N(x_{m-l-1})
\end{smallmatrix}\right)\right. \right)\\
\Psi(Y) &:= \left( \left. \left(\begin{smallmatrix}
\psi_1(x_l)\\
:\\
\psi_N(x_l)
\end{smallmatrix}\right)\right| \ldots \left| \left(\begin{smallmatrix}
\psi_1(x_{m-1})\\
:\\
\psi_N(x_{m-1})
\end{smallmatrix}\right)\right. \right).
\end{align*}
The empirical estimator is then defined similarly to the generator setting via
$$
\tilde{\mathcal{K}}_m^{t} =\tC^{-1} \tA
$$
with 
$$
\tC = \tfrac{1}{m} \Psi(X) \Psi(X)^\top
\qquad\text{and}\qquad
\tA= \tfrac{1}{m} \Psi(X) \Psi(Y)^\top.
$$
We now present the analogue to Theorem~\ref{t:generatorestimate} for the Koopman operator which follows by straightforward adaptations of the results of Section~\ref{ss:projgen}.
\begin{theorem}\label{c:semigroupestimate}
Let Assumption~\ref{as:data} hold. Then, for $t\geq 0$, any error bound $\varepsilon> 0$ and probabilistic tolerance $\delta \in (0,1)$ there is $m_0\in \mathbb{N}$ such that for any $m\geq m_0$,
\begin{align*}
\mathbb{P}\left(\|\mathcal{K}^{t}_\mathbb{V}-\tilde{\mathcal{K}}_m^{t}\|_{F}\leq \varepsilon\right) \geq 1-\delta.
\end{align*}
\end{theorem}
\noindent A sufficient amount of data~$m_0$ can be specified analogously to Theorem~\ref{t:generatorestimate}.}
\section{Extension to control systems}
\label{sec:control}

\MS{\noindent In this section, we derive probabilistic bounds on the approximation error of nonliner control-affine SDE systems of the form
\begin{align}\label{e:sde_control}
\text{d}X_t = \left(F(X_t)+   \sum_{i=1}^{n_c} G_i(X_t)u_i \right)\text{d}t  + \sigma(X_t) \,\text{d}W_t,
\end{align}
with input $u\in\mathbb{R}^{n_c}$ and state $X_t\in\XX$, where $F:\XX\to\mathbb{R}^n$ and $G_i: \XX\to\mathbb{R}^{n}$, $i = 1,\ldots,n_c$, are locally Lipschitz-continuous vector fields. In the deterministic case $\sigma\equiv 0$ the controlled SDE reduces to the control-affine ODE system
\begin{align}\label{e:controlaffine}
\dot{x} = F(x) + \sum_{i=1}^{n_c} G_i(x)u_i.
\end{align}
We will describe how one can apply the bounds on the generators of autonomous (SDE) systems obtained in Section~\ref{sec:SDE} in order to obtain bounds for prediction of control systems, either for i.i.d.\ or ergodic sampling.

%In the ergodic setting, all considerations in this section are also applicable to controlled SDE systems, cf.\ Section~\ref{sec:SDE}, i.e.,
%with \MS{straightforward } modifications. We will stick to the deterministic case for the sake of readability and provide a numerical example of a controlled SDE in Section~\ref{sec:examples}.

Central in this part is the fact that the Koopman generators for control-affine systems are control-affine. More precisely, if $\mathcal{L}^{\bar{u}}$ denotes the Koopman generator for a control-affine system with constant control $\bar{u}\in \mathbb{R}^{n_c}$ and $\bar u = \sum_{i=1}^{r}\alpha_i \bar u_i$ is a linear combination of constant controls $\bar u_i\in\R^{n_c}$, we have
\begin{align}\label{e:gen_ca}
%\mathcal{L}^{\alpha_1\bar{u}_1 + \alpha_2\bar{u}_2} = \mathcal{L}^0 + \alpha_1\left(\mathcal{L}^{\bar{u}_1} -  \mathcal{L}^0\right)+ \alpha_2\left(\mathcal{L}^{\bar{u}_2} -  \mathcal{L}^0\right)
\calL^{\bar u} = \calL^0 + \sum_{i=1}^{n_c}\alpha_i\big(\calL^{\bar u_i} - \calL^0\big).
\end{align}
%for any $\alpha_1,\alpha_2\in\mathbb{R}$ and $\bar{u}_1,\bar{u}_2\in\mathbb{R}^{n_c}$.
This easily follows from the representation \eqref{e:repL} of the Koopman generator, see also \cite[Theorem 3.2]{Peitz2020} for the special (deterministic) case $\sigma\equiv 0$.

We will utilize this property to invoke our results from Section~\ref{sec:SDE} to approximate the Koopman generator corresponding to basis elements of the control space, that is, $\mathcal{L}^{e_i}$, $i=1,\ldots,n_c$, and $\mathcal{L}^0$ corresponding to the drift term to form a bilinear control system in the observables.
%In the following, $\mathcal{O}$ will serve as a placeholder for a Hilbert space of observables $\psi:\mathbb{X}\to\mathbb{R}$ and will be specified later.

Analogously to Assumption~\ref{as:data} we have the following two cases for the collected data and the underlying measure.
\begin{assumption}
	\label{as:control_data}
	Let either of the following hold:
	\begin{itemize}
		\item[(iid)] The data for each autonomous system with control $u=e_i$, $i=0,\ldots,n_c$, is sampled i.i.d.\ from either the normalized Lebesgue measure and contained in a compact set $\mathbb{X}$ or from an invariant measure $\mu_i$ in the sense of Definition~\ref{def:invariant_measure}.
		\item[(erg)] The data for each autonomous system with control $u=e_i$, $i=0,\ldots,n_c$, satisfies Assumption~\ref{as:data}.(erg), i.e., is drawn from a single ergodic trajectory, the probability measure $\mu_i$ of the resulting autonomous SDE is invariant in the sense of Definition~\ref{def:invariant_measure} and the Koopman semigroup is exponentially stable on $L^2_{\mu_i,0}(\mathbb{X})$.
	\end{itemize}
\end{assumption}
\noindent It is important to note that in the first case of (iid), we did not make any assumption of invariance of the set $\mathbb{X}$ for all autonomous systems corresponding to the constant controls $e_i$, $i=0,\ldots,n_c$, as this would be very restrictive. Hence, we have to ensure that the state trajectories remain (with probability one in the stochastic setting~\eqref{e:sde_control}) in the set~$\mathbb{X}$. Sufficient conditions are, e.g., controlled forward invariance of the set~$\mathbb{X}$ or knowing that the initial condition is contained in a suitable sub-level set of the optimal value function of a respective optimal control problem, see, e.g., \cite{BoccGrun14} or \cite{EsteWort20} for an illustrative application of such a technique in showing recursive stability of Model Predictive Control (MPC) without stabilizing terminal constraints for discrete- and continuous-time systems, respectively.
%The price to pay in our results will be that we will only obtain estimates for states that are contained in the set $\mathbb{X}$. However, 
%in particular in the control setting this will naturally be satisfied if one considers, e.g., classical set-point stabilization or model predictive control.

 In the following, we set $\mathcal{O}_i = L^2_{\mu_i}(\mathbb{X})$, $i=1,\ldots,n_c$, and consider the generators $\calL^{e_i}$ in these spaces, respectively. Further, let $\psi_1,\ldots,\psi_N : \mathbb X\to\R$ be $N$ linearly independent observables whose span $\mathbb V = \linspan\{\psi_1,\ldots,\psi_N\}$ satisfies
\begin{align}\label{e:intersec}
\mathbb{V} \subset D(\mathcal{L}^{e_0})\cap D(\mathcal{L}^{e_1})\cap \ldots \cap D(\mathcal{L}^{e_{n_c}}),
\end{align}
where $e_i$, $i=1,\ldots,n_c$, denote the standard basis vectors of $\mathbb{R}^{n_c}$ and $e_0 := 0$. We now discuss two cases of sampling, one corresponding to the approach of Section \ref{sec:SDE} and one to the standard case of i.i.d.\ sampling as in \cite{ZhanZuaz21}.

As the original system and the Koopman generator are control affine, the remainder of this section is split up into two parts. First, we derive error estimates corresponding to autonomous systems driven by $n_c+1$ constant controls. Second, we use these estimates and control affinity to deduce a result for general controls.

\noindent In accordance with the notation in Section \ref{sec:SDE} we define $\calL_{\mathbb V}^{e_i} := P_{\mathbb V}\calL^{e_i}|_{\mathbb V}$ and also use this symbol to denote the matrix representation of this linear operator w.r.t.\ to the basis $\{\psi_1,\ldots,\psi_N\}$ of $\mathbb V$. Its approximation based on the data $x_0,\ldots,x_{m-1}\in\mathbb X$ will be denoted by $\tilde\calL_m^{e_i}$.

\begin{proposition}\label{p:single_estimate}
Let $i \in \{0,\ldots,n_c\}$ be given and Assumption~\ref{as:control_data} hold.
%one of the following two conditions holds:
%\begin{enumerate}
%\item \braces{$\XX=\R^n$, $\sigma\neq 0$} The autonomous SDE \eqref{e:sde_control} with constant control $u\equiv e_i$ gives rise to an invariant measure in the sense of Assumption~\ref{assumption:invariant_measure} such that also Assumption~\ref{as:data}.(erg) holds.
%\item \braces{$\XX$ compact, $\sigma\equiv 0$} The samples are drawn i.i.d.\ from the uniform distribution on $\mathbb X$.
%\end{enumerate}
Then, for any pair consisting of a desired error bound $\varepsilon > 0$ and a probabilistic tolerance $\delta\in (0,1)$, there is a number of data points $m_i$ such that for any $m \geq m_i$, we have the estimate
\begin{align*}
\mathbb{P}\big( \| \LV^{e_i}-\tL^{e_i}\|_F\leq \varepsilon\big) \geq 1-\delta.
\end{align*}
The minimal amount of data $m_i$ is given by the formulas of Theorem~\ref{t:generatorestimate}.
%If Assumption~\textup{1} holds, $m_i$ is given by the formulas of Theorem~\rmref{t:generatorestimate}, i.e.\ Condition \eqref{eq:ineq_m}. If Assumption~\textup{2} holds and the samples are contained in $\mathbb{X}$, $m_i$ is given by the estimate in \textup{\cite[Proposition 4.2]{ZhanZuaz21}}.
\end{proposition}
\begin{proof}
The claim follows immediately from applying Theorem~\ref{t:generatorestimate}.
%In case of Assumption \ref{as:control_data}., the claim follows by applying Theorem~\ref{t:generatorestimate}. In case of Assumption 2, we apply \cite[Proposition 4.2]{ZhanZuaz21} to $\LV^{e_i}$ for $i=0,\ldots,n_c$.
\end{proof}

%In the following, we consider trajectories evolving in the set $\mathbb{X} := \bigcap_{i=0}^{n_c} \mathbb{X}_i$ which is of course not necessarily forward invariant.  

\noindent Having obtained an estimate for the autonomous systems corresponding to the constant controls $e_i, i=0,\ldots n_c$, we can leverage the control affinity of the system to formulate the corresponding results for arbitrary controls. To this end, for any control $u(t) = \sum_{i=1}^{n_c}\alpha_i(t) e_i \in L^\infty(0,T;\mathbb{R}^{n_c})$, we define the projected Koopman generator and its approximation corresponding to the non-autonomous system with control $u$ by
\begin{align*}
\LV^u (t) &:= \LV^0 + \sum_{i=1}^{n_c}\alpha_i(t)\big(\LV^{e_i}-\LV^0\big),\\
\tL^u(t) &:= \tL^0 + \sum_{i=1}^{n_c}\alpha_i(t)\big(\tL^{e_i}-\tL^0\big).
\end{align*}

\begin{theorem}\label{t:coupled_estimate}
Let Assumption~\ref{as:control_data}  hold. 
Then, for any pair consisting of a desired error bound $\tilde\varepsilon > 0$ and probabilistic tolerance $\tilde\delta \in (0,1)$, prediction horizon $T>0$, and control function $u\in L^\infty(0,T;\mathbb{R}^{n_c})$ %such that the solution of \eqref{e:SDE} is contained in $\mathbb{X}$, 
we have
\begin{align*}
\operatorname{ess\,inf}_{t \in [0,T]}\mathbb{P}\big(\|\LV^u(t) - \tL^u(t)\|_F \leq \tilde\varepsilon\big) \geq 1-\tilde\delta,
\end{align*}	
provided that the number~$m$ of data points exceeds $\max_{i=0,\ldots,n_c} m_i$ with $m_i$ defined as in Proposition~\rmref{p:single_estimate} with
$$
\veps = \tfrac{\tilde\varepsilon}{(n_c+1)\left(1  + \sum_{i=1}^{n_c}\|\alpha_i\|_{L^\infty(0,T)}\right)}
\qquad\text{and}\qquad
\delta = 1-\tfrac{\tilde\delta}{n_c+1}.
$$
% there is $m_0 \in \mathbb{N}$ such that for $m\geq m_0$ data points contained in $\mathbb{X}$ we have 
%Further, $m_0$ is given by 
%\begin{align*}
%m &\geq \frac{4C  \xi_{\max}^2 e^{-\omega \Delta_t}}{\eta \lambda_\xi^{\min} (1 - e^{-\omega \Delta_t})^2},
%\end{align*}
%	\begin{align}
%m \geq \frac{2N^2 \max\{\lambda_\xi^{\max},\lambda_\chi^{\max}\}(1 + \eta){\left(\|C^{-1}\| + \tfrac{\varepsilon}{2\|A\|}\right) \|C^{-1}\|}}{ \tfrac{\varepsilon}{2\|A\|}\delta}. 
%\end{align}
%and
%\begin{align}
%m \geq \frac{2N^2 \max\{\lambda_\xi^{\max},\lambda_\chi^{\max}\}(1 + \eta)2\left(\|C^{-1}\|+\tfrac{\varepsilon}{2\|A\|}\right)}{\varepsilon \delta}. 
%\end{align}	
\end{theorem}
\begin{proof}
Again, we omit the subscript of the norm and set $\|\cdot\|=\|\cdot\|_F$.
Using the result of Proposition~\ref{p:single_estimate} and our choice of $m_0$, we have
$$
\mathbb{P}\left( \|\tL^{0}-\LV^0\| \leq \tfrac{\tilde\varepsilon}{(n_c+1)\left(1  + \sum_{i=1}^{n_c}\|\alpha_i\|_{L^\infty(0,T)}\right)}\right) \geq 1-\tfrac{\tilde\delta}{n_c + 1},
$$
and for all $i\in 1,\ldots n_c$
$$
\mathbb{P}\left(\|\LV^{e_i} - \tL^{e_i}\| \leq \tfrac{\tilde\veps}{\left(n_c+1\right)\|\alpha_i\|_{L^\infty(0,T)}}\right) \geq 1- \tfrac{\tilde\delta}{n_c + 1}.
$$
Then we compute for $a.e.\ t\in [0,T]$
\begin{align*}
& \mathbb{P}\left( \| \LV^u(t) -  \tL^u(t) \|\leq \tilde\veps\right)& \\
&\geq \mathbb{P}\left(\left\|\left(1-\sum_{i=1}^{n_c}\alpha_i(t)\right)\left(\LV^0 - \tilde\calL_m^0\right)\right\|+ \sum_{i=1}^{n_c} \left\| \alpha_i(t)\left(\tL^{e_i} - \LV^{e_i}\right)\right\| \leq \tilde\veps\right)\\
&\geq \mathbb{P}\left(\left\|\left(1-\sum_{i=1}^{n_c}\alpha_i(t)\right)\left(\LV^0 - \tilde\calL_m^0\right)\right\| \leq \tfrac{\tilde\veps}{n_c+1}\,\wedge\,
%\bigwedge_{i=1}^{n_c}
\displaystyle\mathop{\mathlarger{\mathlarger{\mathlarger{\forall}}}}_{i=1}^{n_c} : 
\left\|\alpha_i(t) \left(\tL^{e_i} - \LV^{e_i}\right)\right\|
 \leq \tfrac{\tilde\veps}{n_c+1}\right).
\end{align*}
%\red{Auch das verstehe ich nicht (erste Abschaetzung). Es gilt doch
%	\begin{align*}
%		\LV^u(t) -  \tL^u(t)
%		= \left(1-\sum_{i=1}^{n_c}\alpha_i(t)\right)(\calL_V^0 - \tilde\calL_m^0) + \sum_{i=1}^{n_c}\alpha_i(t)(\calL_V^{e_i} - \tilde\calL_m^{e_i})
%\end{align*}}
Next, we use Lemma~\ref{lem:probabilities} from Appendix \ref{ss:technical} with $d = n_c+1$, 
$$
A_0 = \left\{\left\|\left(1-\sum_{i=1}^{n_c}\alpha_i(t)\right)\left(\LV^0 - \tilde\calL_m^0\right)\right\| \leq \tfrac{\tilde\veps}{n_c+1}\right\}
\quad \text{and}\quad
A_i = \left\{ \left\|\alpha_i(t) \left(\tL^{e_i} - \LV^{e_i}\right)\right\| \leq \tfrac{\tilde\veps}{n_c+1}\right\}
$$
for $i=1,\ldots,n_c$. This yields
\begin{align*}
& \mathbb{P}\left( \| \LV^u(t) -  \tL^u(t) \|\leq \tilde\veps\right)& \\
&\geq \mathbb{P}\left( \left\|\left(1-\sum_{i=1}^{n_c}\alpha_i(t)\right)\left(\LV^0 - \tilde\calL_m^0\right)\right\| \leq \tfrac{\tilde\veps}{n_c+1} \right)
%\\&\qquad \qquad+ \sum_{i=1}^{n_c} \mathbb{P}\left( \|\alpha_i(t)\left(\tL^{e_i}-\LV^{{e_i}}\right)\| \leq \tfrac{\tilde\veps}{n_c+1}  \right) - n_c \\
+ \sum_{i=1}^{n_c} \mathbb{P}\left( \|\alpha_i(t)\big(\tL^{e_i}-\LV^{{e_i}}\big)\| \leq \tfrac{\tilde\veps}{n_c+1}  \right) - n_c \\
&\geq \mathbb{P}\left( \|\tL^{0}-\LV^0\| \leq \tfrac{\tilde\veps}{\left(1  + \sum_{i=1}^{n_c} \|\alpha_i\|_{L^\infty(0,T)}\right)(n_c+1)}  \right)
%\\&\qquad + \sum_{i=1}^{n_c} \mathbb{P}\left( \|\tL^{e_i}-\LV^{{e_i}}\| \leq \tfrac{\tilde\veps}{\|\alpha_i\|_{L^\infty(0,T;\mathbb{R})}\left(n_c+1\right)}  \right) - n_c\\
+ \sum_{i=1}^{n_c} \mathbb{P}\left( \|\tL^{e_i}-\LV^{{e_i}}\| \leq \tfrac{\tilde\veps}{\left(n_c+1\right)\|\alpha_i\|_{L^\infty(0,T)}}  \right) - n_c\\
&\geq 1-\tfrac{\tilde\delta}{n_c+1} + \sum_{i=1}^{n_c} \left(1-\tfrac{\tilde\delta}{n_c+1}\right) - n_c = 1-\tilde\delta.
\end{align*}
Taking the essential infimum yields the result.
\end{proof}

\noindent Again, similar as in the previous section, we obtain a bound on trajectories via Gronwall, if the state response is contained in $\mathbb{X}$. %The latter could be omitted, if one assumes forward invariance of the set $\mathbb{X}$ for all controls, cf.\ Remark~\ref{rem:fi}.

\begin{corollary}\label{c:control_trajectory}
Let Assumption~\ref{as:control_data}  hold. Let $T,\veps>0$ and $\delta\in (0,1)$, $z_0\in \mathbb{R}^N$ and $u\in L^\infty(0,T;\mathbb{R}^{n_c})$ such that the solution of \eqref{e:SDE} is contained in $\mathbb{X}$ with probability one. 
Then there is $m_0\in \mathbb{N}$ such that for $m\geq m_0$ %data points contained in $\mathbb{X}$, 
the solutions $z,\tilde{z}$ of
%for $m_0 \in \mathbb{N}$ as given in Theorem~\ref{t:coupled_estimate} and $m\geq m_0$ data points contained in $\mathbb{X}$, the solutions $z,\tilde{z}$ of
\begin{align*}
&&\dot{z}(t) &= \LV^u(t)z &&z(0)=z_0\\
&&\dot{\tilde{z}}(t) &= \tL^u(t)\tilde{z} &&\tilde{z}(0)={z}_0
\end{align*}
satisfy
\begin{align*}
\min_{t\in [0,T]}\mathbb{P}\big( \|z(t)-\tilde{z}(t)\|_2 \leq \varepsilon\big) \geq 1-\delta.
\end{align*}
\end{corollary}
\begin{proof}
See Appendix~\ref{s:errorbound_traj}.
\end{proof}

\noindent As in Corollary~\ref{c:trajest}, $m_0$ can explicitly be computed by combining Theorem~\ref{t:coupled_estimate} with the constants in Gronwalls inequality.

We conclude this section with a final corollary regarding the optimality of the solution obtained using an error-certified Koopman model. To this end, we consider the optimal control problem with $x_0\in \mathbb{X}$ and a stage cost $\ell:\R^n\times \R^{n_c} \to \mathbb{R}$:
\begin{equation}\label{eq:OCP_full}
\begin{aligned}
\min_{u\in L^\infty(0,T;\mathbb{R}^{n_c})} &\int_0^T \ell(x(t),u(t))\,\text{d}t\\
\mbox{s.t.}\qquad \dot{x} = &F(x) + \sum_{i=1}^{n_c} G_i(x)u_i, \qquad x(0)=x_0.
\end{aligned}
\end{equation}
In what follows, we compare the optimal value of the Koopman representation of \eqref{eq:OCP_full} projected onto the subspace of observables $\mathbb{V}$ with initial datum $z_0 = \Psi (x_0)$
\begin{equation}\label{eq:OCP_full_Koop}
\begin{aligned}
\min_{\alpha\in L^\infty(0,T;\mathbb{R}^{n_c})} &\int_0^T \ell(P(z(t)),\alpha(t))\,\text{d}t\\
\mbox{s.t.}\qquad \dot{{z}}(t) = &\left[\LV^0 + \sum_{i=1}^{n_c}\alpha_i(t)\left(\LV^{e_i}-\LV^0\right)\right]{z}(t), \qquad {z}(0)={z}_0,
\end{aligned}
\end{equation}
to the optimal value of the surrogate-based control problem:
\begin{equation}\label{eq:OCP_surrogate}
\begin{aligned}
\min_{\tilde \alpha\in L^\infty(0,T;\mathbb{R}^{n_c})}& \int_0^T \ell(P(\tilde{z}(t)),\tilde{\alpha}(t))\,\text{d}t\\
\mbox{s.t.}\qquad \dot{\tilde{z}}(t) = &\left[\tL^0 + \sum_{i=1}^{n_c}\tilde\alpha_i(t)\left(\tL^{e_i}-\tL^0\right)\right]\tilde{z}(t), \qquad \tilde{z}(0)={z}_0,
\end{aligned}
\end{equation}
where $P$ maps a trajectory of observables to a trajectory in the state space, which in practice is frequently realized by including the coordinates of the identity function in the dictionary~$\Psi$ of observables.%contained in $\mathbb{V}$, i.e., $\Phi (z_1,\ldots,z_N) = \sum_{i=1}^N z_i \psi_i.$%

%\red{$P$ sollte hier meiner Meinung nach einfach der Isomorphismus $\Phi : \R^N\to\mathbb V$ sein und kann somit als Einbettung $\R^N\to L^2_\mu$ gesehen werden -- also genau das Gegenteil einer Projektion.}

\begin{corollary}\label{c:control_problem}
Let $T,\veps>0$, $\delta\in (0,1)$, $z_0\in \mathbb{R}^N$, let $J$ be locally Lipschitz continuous and let Assumption~\ref{as:control_data} hold.
Furthermore, let $(z^*,\alpha^*)$ be an optimal solution of problem \eqref{eq:OCP_full_Koop} such that the state response of \eqref{eq:OCP_full} emanating from the control $\alpha^*$ is contained in $\mathbb{X}$. Then there is $m_0\in \mathbb{N}$ such that for $m\geq m_0$ data points contained in $\mathbb{X}$, there exists a tuple $(\tilde{z},\tilde{\alpha})$ which is feasible for \eqref{eq:OCP_surrogate} such that for the cost, we have the estimate
\[
\mathbb{P}\left(\left\vert\int_0^T\ell(P(\tilde{z}(t)),\tilde{\alpha}(t)) - \ell(P(z^*(t)),\alpha^*(t))\,\mathrm{d}t)\right\vert \leq \varepsilon\right) \geq 1-\delta.
\]
\end{corollary}}
\section{Numerical examples}\label{sec:examples}

In this section, we first present numerical experiments on the derived error bound for the Koopman generator, and then discuss the implications for optimal control. In particular, we emphasize that the bilinear Koopman model from Section \ref{sec:control} appears to be the best approach for a straightforward transfer of predictive error bounds to the control setting.

\MS{\subsection{Generator Error Bounds: Ornstein-Uhlenbeck Process}
\label{subsec:ExampleErrorBounds}
We begin by investigating the validity and accuracy of the error bounds for the Galerkin matrices of a single SDE system, as derived in Proposition~\ref{thm:prob_error_bounds_xi}. To this end, we consider the one-dimensional reversible Ornstein-Uhlenbeck (OU) process
\begin{equation}
\label{eq:ou_sde}
\mathrm{d}X_t = - X_t \mathrm{d}t + \mathrm{d}W_t.
\end{equation}

\begin{figure}[bht]
\centering
\includegraphics[width=.51\columnwidth]{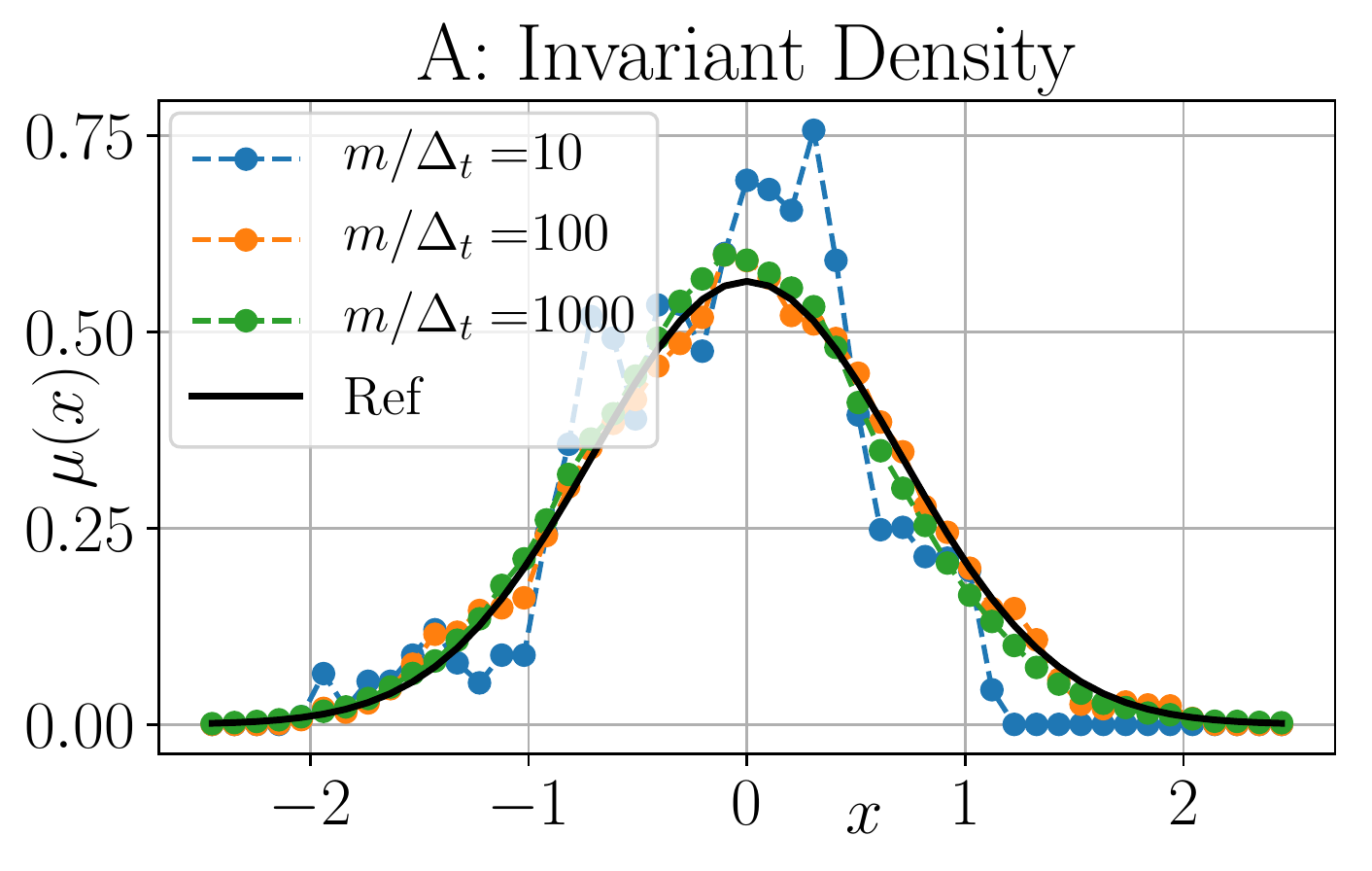}
\includegraphics[width=.49\columnwidth]{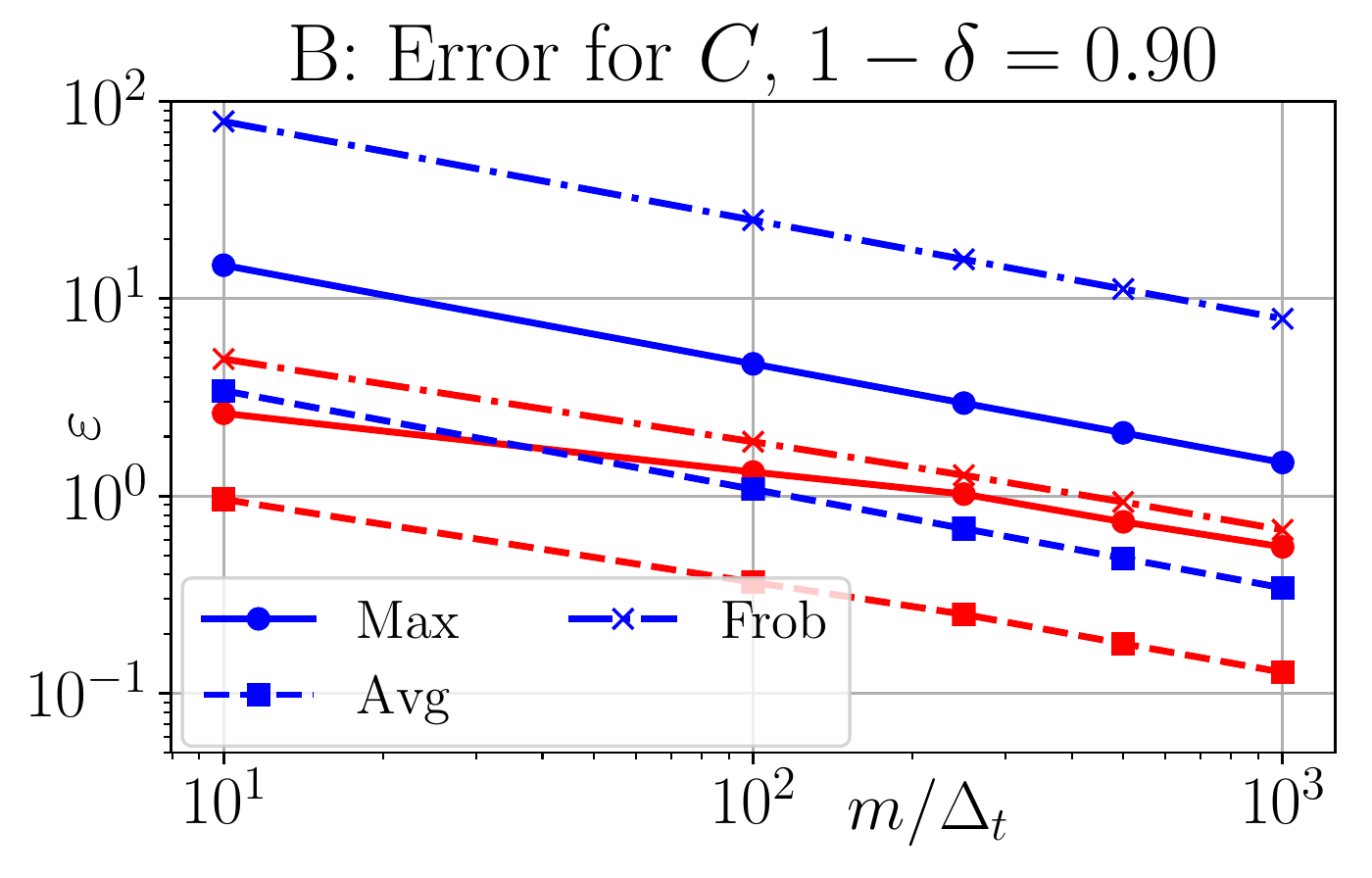}
\includegraphics[width=.49\columnwidth]{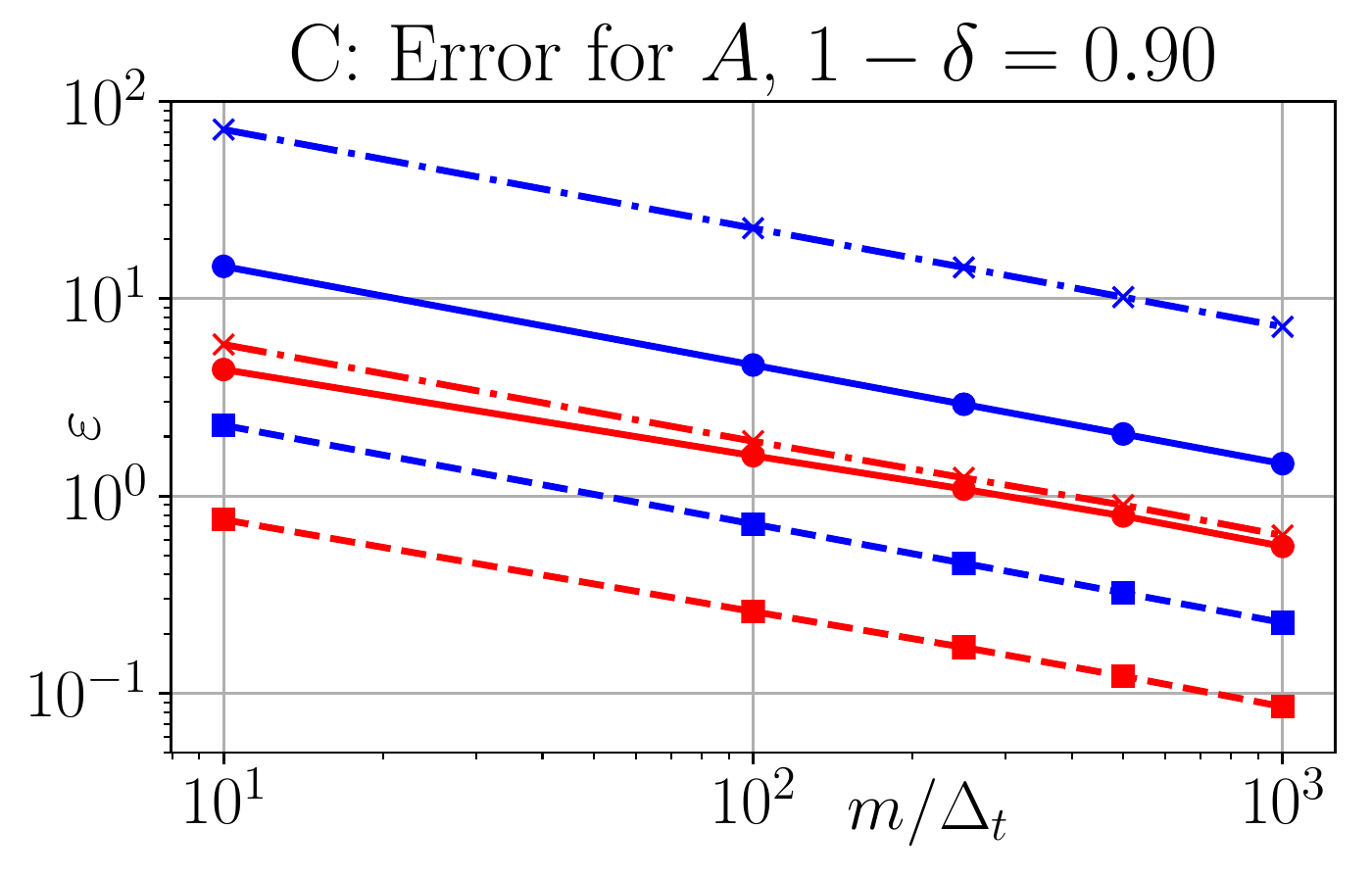}
\caption{\MS{Numerical Results for one-dimensional OU Process~\eqref{eq:ou_sde}. A: Exact invariant density $\mu$ in black, compared to histograms of the first $m$ points of an exemplary trajectory, for various data sizes $m$. B: Error bounds for $C$ corresponding to confidence level $1 - \delta = 0.9$. We show both the theoretical estimates obtained in Proposition~\ref{thm:prob_error_bounds_xi} (blue), as well as the data-based estimates obtained as described in the text (red). We show the maximal error over all entries $C_{ij}$ (dots), the average error over all matrix entries (squares), and the Frobenius norm errors $\|\tilde{C}_m -  C\|_F$. C: The same as B for the matrix $A$. \label{fig:ou_process}}}
\end{figure}
\noindent As the spectrum of the generator $\mathcal{L}$ of the OU process, as well as its invariant density, are known in analytical form, we can exactly calculate the Galerkin matrices $C, \, A$, all variances $\sigma^2_{\Phi_{ij}}$, and asymptotic variances $\sigma^2_{\Phi_{ij}, \infty}$, if we consider a basis set comprised of monomials, see Appendix~\ref{app:ou_analytical}.

We consider monomials of maximal degree four (i.e. $N = 4$), and set the discrete integration time step to $\Delta_t = 10^{-3}$. For a range of different data sizes $m$ and confidence levels $\delta$, we estimate the minimal error $\varepsilon$ that can be achieved with probability $1 - \delta$ for a variety of quantities of interest. We calculate $\varepsilon$ for all individual entries $C_{ij}$ and $A_{ij}$ using inequality~\eqref{eq:error_phi_ij_rev}. Moreover, we also calculate $\varepsilon$ for the Frobenius norm errors in $C$ and $A$ by means of~\eqref{eq:ineq_m_rev}.

In order to compare our bound to the real error, we conduct 500 identical experiments. For each experiment, we generate an independent simulation of the OU process~\eqref{eq:ou_sde}, with initial condition drawn from the invariant distribution. For each trajectory and each of the data sizes $m$ considered, we estimate the matrices $\tilde{C}_m, \, \tilde{A}_m$. We then calculate the absolute entry-wise errors to $C$ and $A$, as well as the Frobenius norm errors $\|\tilde{C}_m - C\|_F$ and $\|\tilde{A}_m - A\|_F$. Finally, we numerically compute the $1 - \delta$-percentile of each of these errors for all confidence levels $\delta$ considered above (i.e., the error $\varepsilon$ below which $450$ of the $500$ repeated experiments lie). These can be directly compared to the probabilistic bounds $\varepsilon$ obtained from our theoretical estimates.

The results are shown in Figure~\ref{fig:ou_process}. We can see in panels B and C that our estimates for individual entries of the Galerkin matrices $C$ and $A$ are quite accurate, as the data-based error is over-estimated by only a factor of two to three. Our estimates for Frobenius norm errors are less accurate, with approximately one order of magnitude difference between theoretical and data-based errors. It can be concluded that the factor  $N$ in~\eqref{eq:ineq_m_rev} is too coarse in this example, as the actual Frobenius norm error only marginally exceeds the maximal entry-wise error. Nevertheless, the qualitative behaviour of all theoretical error bounds is confirmed by the data.}

\subsection{Extension to control systems}
In this section, we illustrate our findings for deterministic as well as stochastic systems regarding prediction and control. We compare the solution of the exact model to the bilinear system 
\begin{equation}\label{eq:discreteBilinearModel}
\begin{aligned}
%		\widehat{z}_{i+1} &= \left(K_0 + \sum_{j=1}^{n_c}(K_j-K_0) u_{j,i}\right)z_i \\
%		z_{i+1} &= \psi(P(\widehat{z}_{i+1})) \\
%		z_0 &= \psi(x(t_0)),
\widehat{z}(t) &= \psi(P({z}(t))) \\
%	\dot{z}(t) &= \left(\mathcal{L}^0 + \sum_{j=1}^{n_c}(\mathcal{L}^{e_j}-\mathcal{L}^0) u_{j}(t)\right)\widehat{z}(t) \\
\dot{z}(t) &= \left[\tL^0 + \sum_{i=1}^{n_c}\alpha_i(t)\left(\tL^{e_i}-\tL^0\right)\right]\widehat{z}(t)\\
z(t_0) &= \psi(x(t_0)),
\end{aligned}
\end{equation}
where $n_c$ is the dimension of the control input $u$, and $P$ is the projection of the lifted state $z$ onto the full state $x \in \mathbb{X}$. Note that the first line, i.e., the \textit{project-and-lift} step is not required if the space $\mathbb{V}$ spanned by the $\{\{\psi_k\}_{k=1}^N\}$ is a \textit{Koopman-invariant subspace} \cite{PBK18}. Moreover, it becomes less and less important the more the dynamics of the $\tL$ are truly restricted to $\mathbb{V}$, or -- alternatively -- if we are not interested in long-term predictions, for instance in the MPC setting.
Besides the bilinear model \eqref{eq:discreteBilinearModel}, we also compare the true solution to the linear model obtained via eDMD with control, see \cite{Proctor2016,KM18a} for details.
Optimality of the computed trajectories from a theoretical standpoint will not be addressed here, as the error bounds for $\tL$ are still too large. However, the principled approach is to choose an $m$ such that Corollary \ref{c:control_problem} holds.

For the numerical discretization, we use eDMD with a finite lag time to obtain a discrete-time version of \eqref{eq:discreteBilinearModel} in case of the Duffing system, which corresponds to an explicit Euler discretization \cite{Peitz2020}. For the Ornstein-Uhlenbeck example, we calculate the generator using gEDMD \cite{Klus2020} and then obtain the resulting discrete-time version by taking the matrix exponential. In the case of eDMD with control, we use the standard algorithm from \cite{KM18a}, which also results in a forward Euler version of the linear system $\dot{z} = \hat A z + \hat B u$, i.e.,
\begin{equation}\label{eq:DMDcModel}
\begin{aligned}
\widehat{z}_{i+1} &= A z_i + B u_i, \\
z_{i+1} &= \psi(P(\widehat{z}_{i+1})), \\
z_0 &= \psi(x(t_0)),
\end{aligned}
\end{equation}
where we have again added the \textit{project-and-lift} step necessary for high prediction accuracy over long time horizons.

\subsubsection{Duffing equation (ODE)}
The first system we study is the Duffing oscillator:
\begin{equation}\label{eq:Duffing_NonlinCon}
\begin{aligned}
\tfrac{\text{d}x}{\text{d}t} = \begin{pmatrix}
x_2 \\ -\delta x_2 - \alpha x_1 - 2\beta x_1^3 u
\end{pmatrix}, \quad x(t_0) = x_0.
\end{aligned}
\end{equation}
with $\alpha = -1$, $\beta = 1$ and $\delta = 0$.
Note that the control does not enter linearly, which is a well-known challenge for DMDc \cite{Peitz2020}.

As the dictionary $\psi$, we choose monomials with varying maximal degrees, and we also include square and cubic roots for comparison. For the data collection process, we simulate the system with constant control inputs $u=0$ and $u=1$ using the standard Runge-Kutta scheme of fourth order with time step $h=0.005$. As the final time, we choose $T = n_{lag} h$ seconds, where $n_{lag}$ is the integer number of time steps we step forward by the discrete-time Koopman operator model. We perform experiments for both $n_{lag}=1$ and $n_{lag}=10$. Each trajectory yields one tuple $(x,y) = (x(0), x(T))$, and we sample various numbers $m$ of data points with uniformly distributed random initial conditions over the rectangle $[-1.5,1,5]^2$.
\begin{figure}[h!]
\centering
\includegraphics[width=.7\columnwidth]{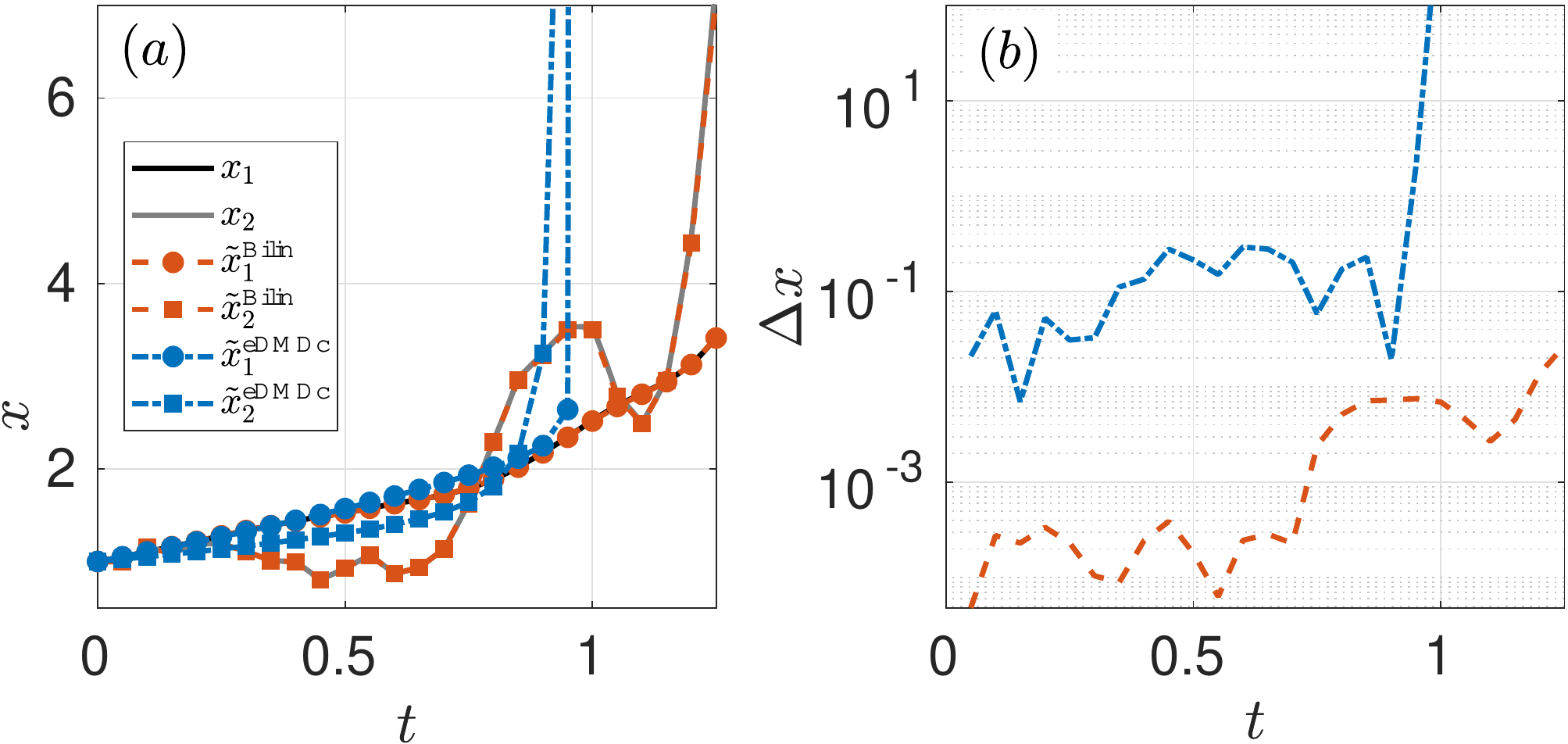}
\caption{Comparison of ODE solution, the bilinear surrogate model and the linear model obtained via eDMDc for the system \eqref{eq:Duffing_NonlinCon} for a random control input with $u(t) \in[-1,1]$.}% (a) Comparison of the trajectories ($x_1$: --- and $x_2$: - - -). (b) The corresponding error $\Delta x$.}
\label{fig:Duffing_Prediction}
\end{figure}

Fig.\ \ref{fig:Duffing_Prediction} shows the prediction accuracy for $m=100$ and $n_{lag}=10$, where excellent agreement is observed for the bilinear surrogate model. In particular the relative error
\[
\Delta x(t) = \frac{\|x(t) - \tilde{x}(t)\|_2}{\|x(t)\|_2},
\]
where $\tilde{x}(t)= P(z(t))$ is the solution obtained via the surrogate model, is below 0.1 percent for almost 3 seconds, whereas the eDMDc approach has a large error of $\approx 10\%$ from the start and becomes unstable within the first second. 

To study the influence of the size of the training data set, Fig.\ \ref{fig:Duffing_Error} shows boxplots of the one-step prediction accuracy for various $m$. Each boxplot was obtained by performing 20 trainings of a bilinear system according to the procedure described above. After each training, a single time step was made with $1000$ uniformly drawn random initial conditions $x_0 \in [-1.5,1,5]^2$ control inputs $u \in [0,1]$, both. Consequently, each boxplot consists of $2\cdot 10^4$ data points. We see that, as expected, the training error decreases for larger $m$. However, what is really surprising is that a saturation can be observed already at $m=30$ for an ODE system. Beyond that, no further improvement can be seen, which demonstrates the advantage of (i) the linearity of the Koopman approach and (ii) the usage of autonomous systems for the model reduction process.
\begin{figure}[thb]
\centering
\includegraphics[width=1\columnwidth]{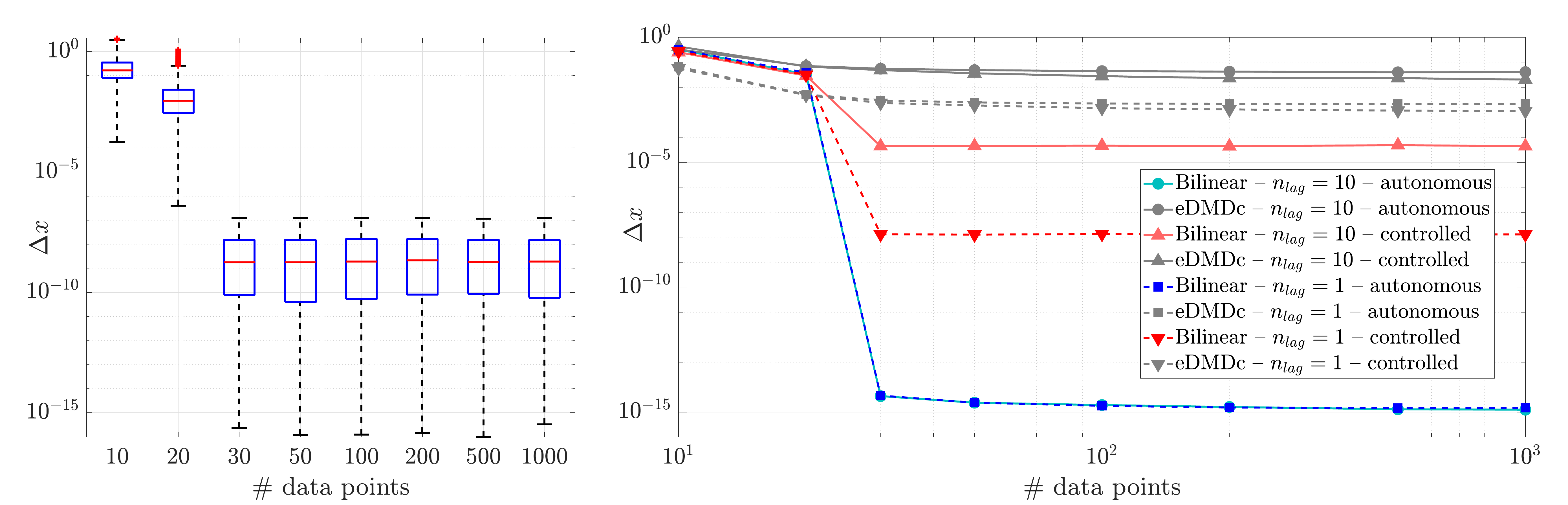}
\caption{\MS{Left: Boxplot of the relative one-step prediction error over 20 training runs and $1000$ different samples $(x_0,u)$ in each run for a dictionary of monomials up to degree at most five and $n_{lag}=1$. Right: The influence of the lag time as well as the control input on the mean accuracy (the dashed line with triangle symbols corresponds to the left plot). We see that the lag time plays an important role in the control setting.}}
\label{fig:Duffing_Error}
\end{figure}

Interestingly, the lag time between two consecutive data points has a critical impact on the maximal accuracy in the control case. This is due to the fact that the bilinear surrogate model is only exact for the Koopman generator \cite{Peitz2020}. For a finite lag time, the bilinear model is a first order approximation such that smaller lag times are advantageous. Nevertheless, the accuracy still significantly supersedes the eDMDc approach.

Another interesting observation can be made with respect to the choice of the dictionary $\psi$. Fig.\ \ref{fig:Duffing_Error_Means} shows a comparison of the mean errors (analogous to the red bars in Fig.\ \ref{fig:Duffing_Error} for various dictionaries. We observe excellent performance for monomials with degree three or larger. The addition of roots of $x$ is not beneficial at all, and in particular, smaller dictionaries are favorable in terms of the data requirements, which is in agreement with our error analysis and which was also reported in \cite{PK20}. 
\begin{figure}[thb]
\centering
\includegraphics[width=.7\columnwidth]{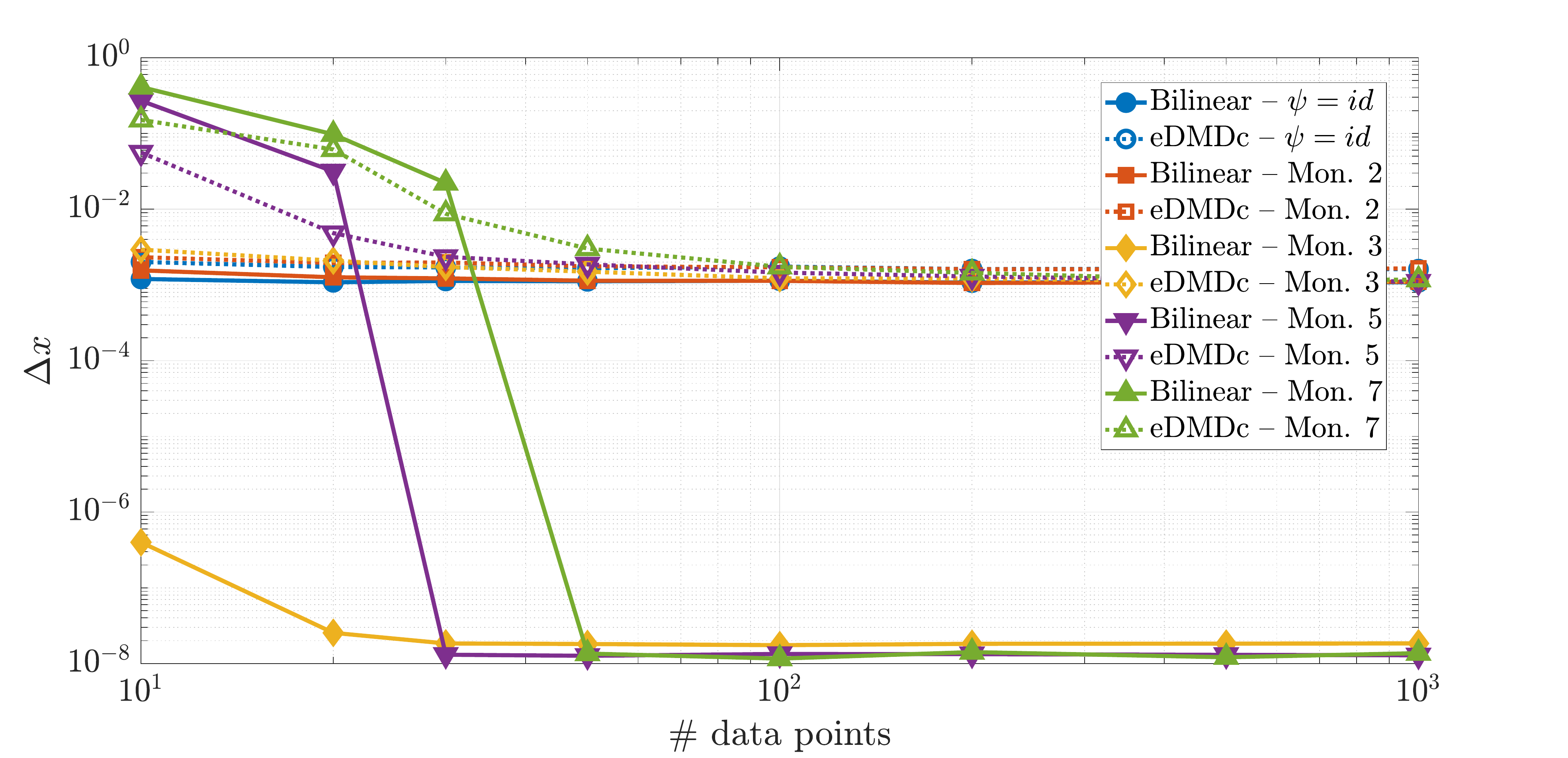}
\caption{\MS{Mean relative one-step prediction errors for various dictionaries and data set sizes $m$.}}
\label{fig:Duffing_Error_Means}
\end{figure}

\noindent Next, we study the stabilization of the system \eqref{eq:Duffing_NonlinCon} for the final time $T=5$. Using the time discretization as above and a straight-forward single-shooting method, this yields a 100-dimensional optimization problem similar to Problem \eqref{eq:OCP_surrogate} from Corollary \ref{c:control_problem}:
\begin{equation}\label{eq:OCP}
\begin{aligned}
\min_{u} \int_{0}^5 &\|P (z(t)) - x^{\mathrm{ref}}(t)\|^2 \\
\mbox{s.t.} \qquad &\eqref{eq:discreteBilinearModel}
\end{aligned}
\end{equation}
where $x^{\mathrm{ref}}$ is the reference trajectory to be tracked. Fig.\ \ref{fig:Duffing_Control} demonstrates the performance for $x^{\mathrm{ref}}=0$ with models that were obtained using only $m=25$ training samples for each of the Koopman approximations, where almost perfect agreement with the solution using the full system is achieved. In contrast, the eDMDc approximation fails for System \eqref{eq:Duffing_NonlinCon}, even when initializing with the optimal solution from the full system.

\begin{figure}[thb]
\centering
\includegraphics[width=.7\columnwidth]{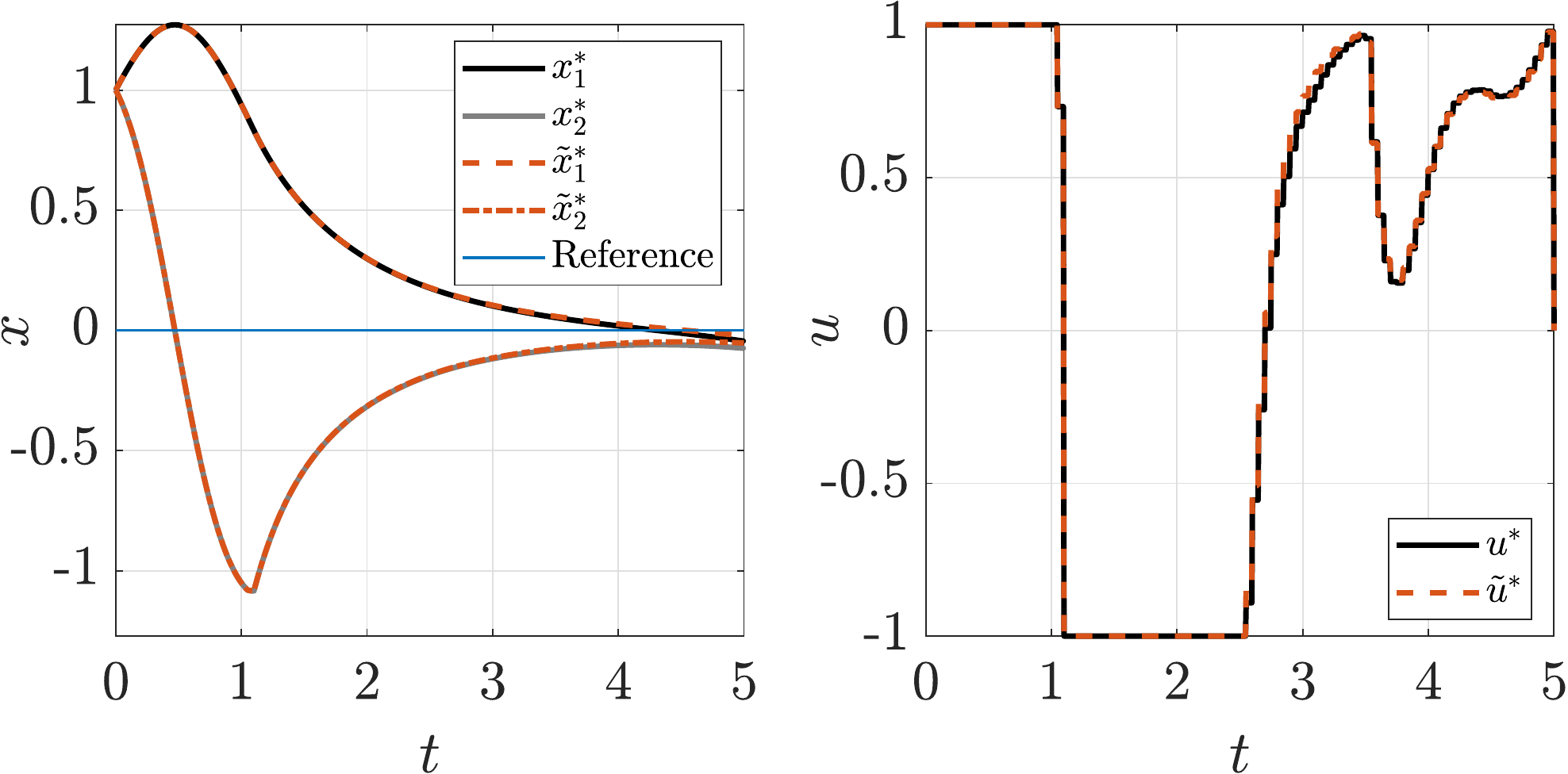}
\caption{Control performance using the true ODE model (black) and the bilinear surrogate model (orange). The results are almost indistinguishable, whereas eDMDc fails.}
\label{fig:Duffing_Control}
\end{figure}

\subsubsection{Ornstein-Uhlenbeck process (SDE)}
For the stochastic setting, we consider an Ornstein-Uhlenbeck process with a control input:
\begin{equation}\label{eq:OU_NonlinCon}
\mathrm{d}X_t = -\alpha (u X_t) \mathrm{d}t + \sqrt{2 \beta^{-1}} \mathrm{d}W_t.
\end{equation}
with $\alpha = 1$, $\beta = 2$ and $u(t) \in [0,1]$. 
The system is simulated numerically using an Euler-Maruyama integration scheme with a time step of $10^{-3}$ as in Section \ref{subsec:ExampleErrorBounds}.
For both systems, we calculate the Koopman operator corresponding to $u=0$ and $u=1$, respectively, using the gEDMD procedure presented in \cite{Klus2020} with monomials up to degree five. We then calculate the corresponding Koopman operators for the time step $h=0.05$ using the matrix exponential.
\begin{figure}
\centering
\includegraphics[width=.65\columnwidth]{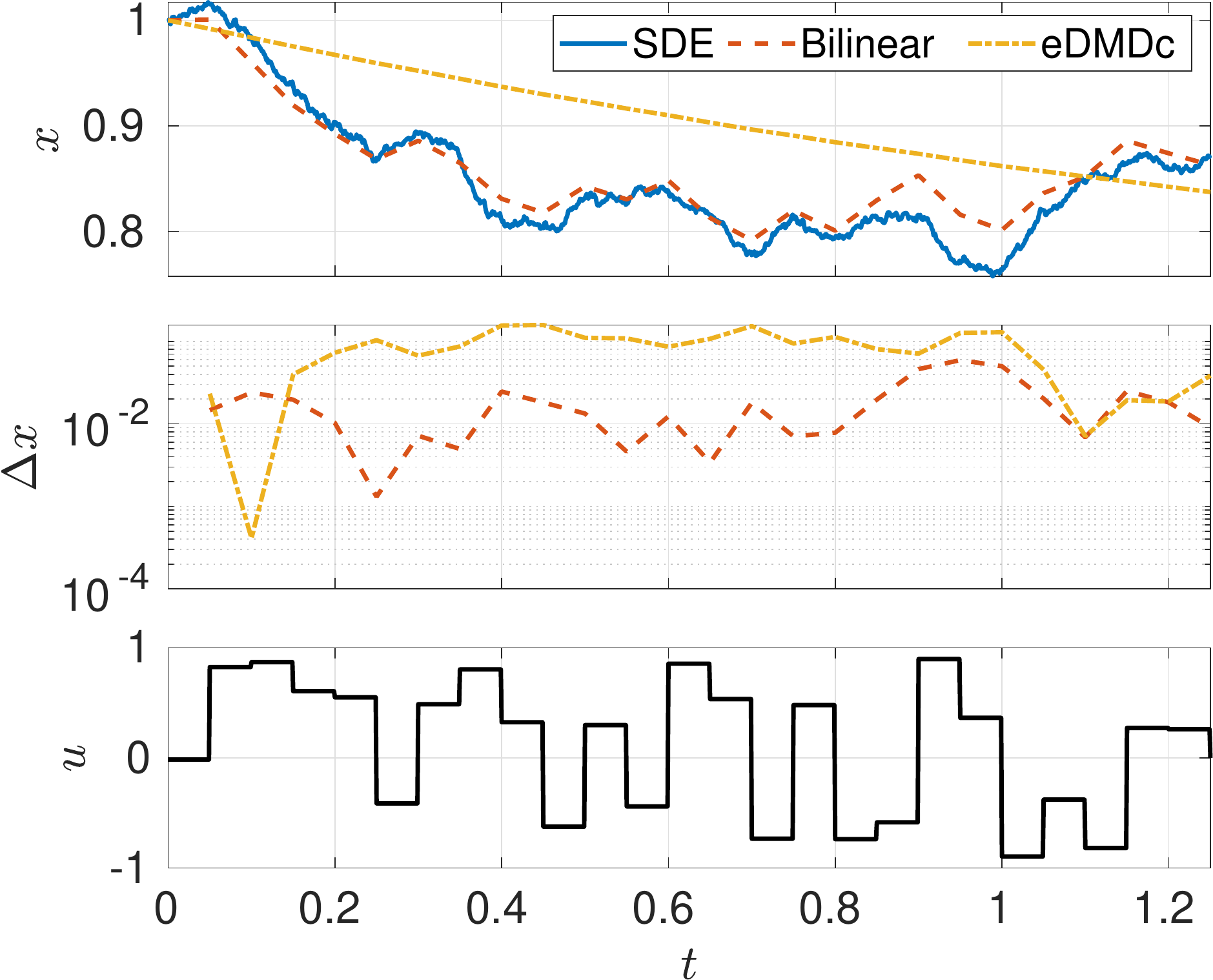}
\caption{Prediction accuracy for the expected value of the Ornstein-Uhlenbeck process (approximated by averaging over 100 simulations) of the bilinear system and eDMDc, respectively.}
\label{fig:OU_Prediction}
\end{figure}

To study the prediction performance (cf.\ Fig.\ \ref{fig:OU_Prediction}), we proceed in the same way as for the Duffing system, except that we now compare the expected values, approximated by averaging over 100 SDE simulations. The results are very similar to the deterministic case, where the performance of both surrogate modeling techniques is comparable when the control enters linearly, and very poor for eDMDc otherwise. Even though the Ornstein-Uhlenbeck process is stochastic, the linearity is highly favorable for the data requirements. We do not observe any considerable deterioration even in the very low data limit.

Finally, in the control setting, we aim at tracking the expected value $\mathbb{E}[X_t]$, which is precisely the quantity that is predicted by the Koopman operator. Thus, Problem \eqref{eq:OCP} can directly be applied to SDEs as well. In order to compare the results to the full system, we average over 20 simulations in the evaluation of the objective function value when using the SDE. However, this appears to be insufficient, as the performance is inadequate, cf.\ Fig.~\ref{fig:OU_Control}. The bilinear surrogate model, on the other hand, shows very good performance with a small amount of $m=100$ training data points. %\textcolor{red}{Ich muss die Resultate für eDMDc nochmal checken, das sieht mir zumindest im linearen Fall nicht gut genug aus.}
\begin{figure}[htb]
\centering
\includegraphics[width=.65\columnwidth]{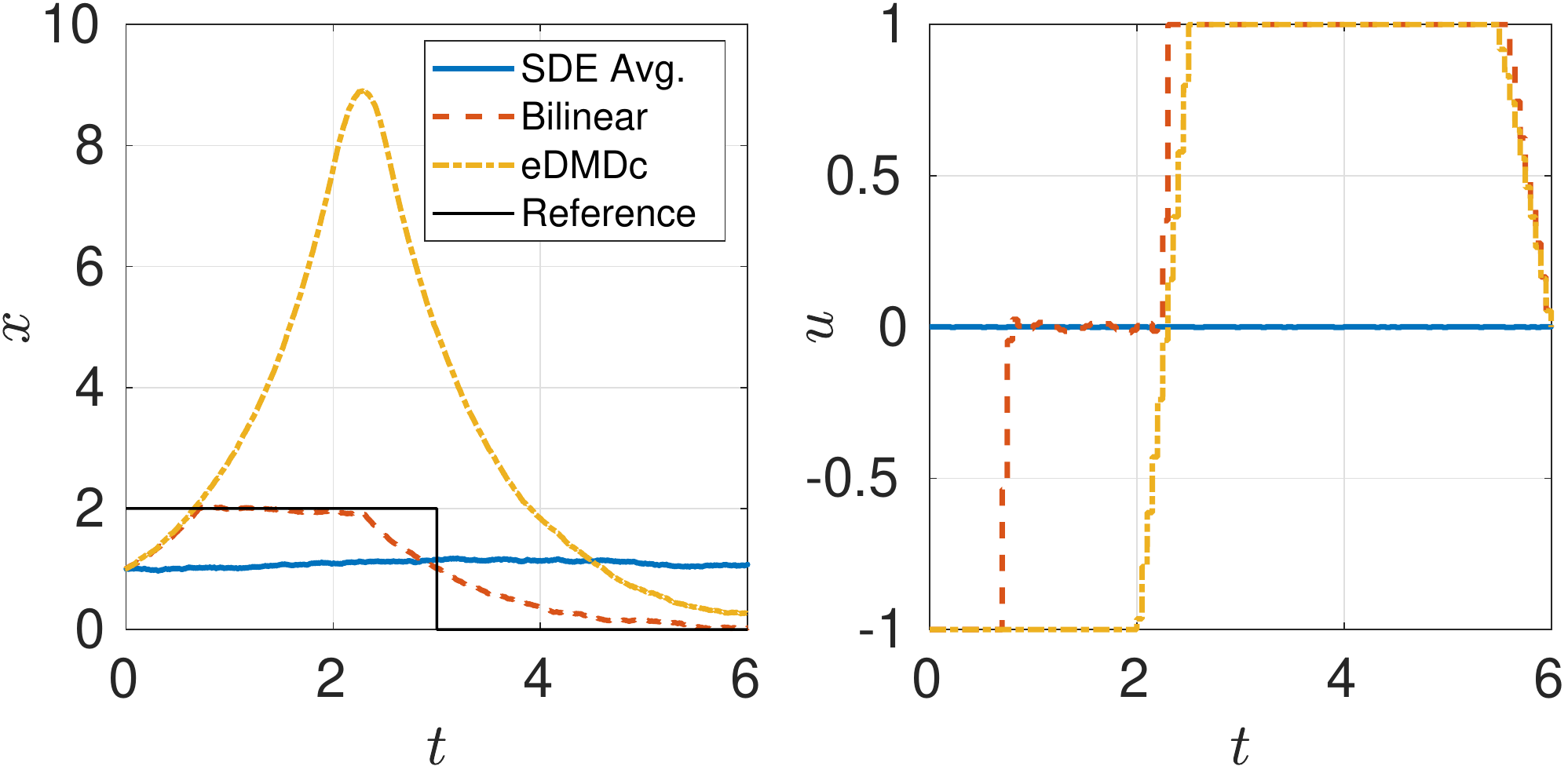}
\caption{Control of the expected value of the Ornstein-Uhlenbeck process (approximated by averaging over 100 simulations using the optimal control input shown in the bottom plots). In the SDE-based control, we have used 20 simulations in each objective function evaluation.}
\label{fig:OU_Control}
\end{figure}

\section{Conclusions} % \& Outlook}
\label{sec:conclusion}

\noindent
We presented the first rigorously derived probabilistic bounds on the finite-data approximation error for the Koopman generator of SDEs and nonlinear control systems. Furthermore, by using slightly more advanced techniques from probability theory, we also relaxed the assumption of i.i.d. data invoked in~\cite{ZhanZuaz21} in the ODE setting. Moreover, we also provided an analysis for the error propagation to estimate the prediction accuracy in terms of the data size. A novelty for SDEs and in the control setting is that our bounds explicitly depend on the number of data points (and not only in the infinite-data limit). Further, the proposed techniques provide the theoretical foundation for the Koopman-based approach~\cite{Peitz2020} to control-affine systems, which seems to be superior for control and particularly well-suited for MPC, since it avoids the curse of dimensionality w.r.t.\ the control dimension. 

\bibliographystyle{abbrv}
\bibliography{references}

\appendix
\section{Appendix}
\subsection{Norm of the isomorphism $\mathbb{V}\simeq \mathbb{R}^n$}\label{s:isomorphism}

\begin{proposition}\label{p:normeq}
Let $\mathbb{V} = \operatorname{span}\{\{\psi_j\}_{j=1}^N\}\subset \Lmu$, $\mathcal{B} \in L(\mathbb{V},\mathbb{V})$ and $B\in \mathbb{R}^{n\times n}$ be its corresponding matrix representation.
Then
$$
\sqrt{\tfrac{\lambda^{\min}(C)}{\lambda^{\max}(C)}}\|B\|_2 \leq \|\mathcal{B}\|_{L(\mathbb V,\mathbb V)} \leq \sqrt{\tfrac{\lambda^{\max}(C)}{\lambda^{\min}(C)}} \|B\|_2 
$$
where $C_{i,j} = \langle \psi_i,\psi_j\rangle_{\Lmu}$.
\end{proposition}
\begin{proof}
This follows from the identity 
$$
\left\| \sum_{i=1}^{n} \alpha_i \psi_i\right\|_{\Lmu}^2 = \sum_{i,j=1}^N \alpha_i \alpha_j \langle \psi_i,\psi_j\rangle_{\Lmu} = \alpha^\top C\alpha,
$$
which shows the equivalence of the vector norms. This induces the desired equivalence of the operator norms.
\end{proof}

\subsection{A technical lemma}\label{ss:technical}
\begin{lemma}
\label{lem:probabilities}
Let $A_i$, $i=1,\ldots,d$, be measurable sets. Then
\begin{align*}
\mathbb{P}\left(\bigcap_{i=1}^d A_i \right) = \sum_{i=1}^d \mathbb{P}(A_i) - \sum_{i=1}^{d-1} \mathbb{P}\left(A_i \cup \bigcap_{j=i+1}^d A_j\right).
\end{align*}
Moreover, if $\mathbb{P}\left(A_i\right) \geq 1-\delta$ for all $i=1,\ldots,d$, then 
\begin{align*}
\mathbb{P}\left(\bigcap_{i=1}^d A_i \right)  \geq 1-d\delta.
\end{align*}
\end{lemma}
\begin{proof}
Inductively applying the classical formula
\begin{align*}
\mathbb{P}\left(A_1\cap A_2\right) = \mathbb{P}(A_1) + \mathbb{P}(A_2) - \mathbb{P}(A_1\cup A_2)
\end{align*}
yields
\begin{align*}
\mathbb{P}\left(\bigcap_{i=1}^d A_i \right) &= \mathbb{P}\left(A_1 \cap \bigcap_{i=2}^d A_i\right) = \mathbb{P}(A_1)  + \mathbb{P}\left(\bigcap_{i=2}^d A_i\right) - \mathbb{P}\left(A_1 \cup \bigcap_{i=2}^d A_i\right)\\
&=\sum_{i=1}^d \mathbb{P}\left(A_i\right) - \sum_{i=2}^{d-1}\mathbb{P}\left(A_i \cup \bigcap_{j=i+1}^dA_j\right) - \mathbb{P}\left(A_1 \cup \bigcap_{i=2}^d A_i\right)\\
&= \sum_{i=1}^d \mathbb{P}\left(A_i\right) - \sum_{i=1}^{d-1}\mathbb{P}\left(A_i \cup \bigcap_{j=i+1}^dA_j\right),
\end{align*}
which proves the first claim. The second claim follows by estimating the first sum by $d(1-\delta)$ from below, and the second sum by $-(d-1)$ from below.
\end{proof}

\subsection{Proof of the error bound on the trajectories}\label{s:errorbound_traj}
\begin{lemma}
\label{lem:gronwall}
Let $z$ and $\tilde{z}$ solve \eqref{e:z} and \eqref{e:tz} respectively. Then for all $t\geq 0$
\begin{align*}
\|z(t)&-\tilde{z}(t)\|_2 
\leq \|\tL - \LV\|_2\|\tilde z\|_{L^1(0,t;\mathbb{R}^N)}e^{t \|\LV\|_2}
\end{align*}
\end{lemma}
\begin{proof}
Denoting $e = z-\tilde{z}$, subtracting \eqref{e:tz} from \eqref{e:z} and integrating over a time interval $[0,t]$ with $t\geq 0$ we obtain that
\begin{align*}
e(t) &= \int_0^t \LV z(s) - \tL \tilde{z}(s)\,\text{d}s\\
&= \int_0^t \LV e(s) - \left(\tL - \LV\right)\tilde{z}(s)\,\text{d}s 
\end{align*}
This implies using Gronwalls inequality, cf.\ \cite[Theorem 2.1]{Chicone2006}, that
\begin{align*}
\|e(t)\|_2 &\leq \int_0^t \|\LV\|_2 \|e(s)\|_2 + \|\tL - \LV\|_2 \|\tilde{z}(s)\|_2\,\text{d}s\\
&\leq e^{t \|\LV\|_2} \int_0^t\|\tL - \LV\|_2\|\tilde z(s)\|_2\,ds\\
&= e^{t \|\LV\|_2} \|\tL - \LV\|_2\|\tilde z\|_{L^1(0,t;\mathbb{R}^N)}.
\end{align*}
\end{proof}

\begin{myproof}[Corollary~\ref{c:trajest}]
Using the bound of Lemma~\ref{lem:gronwall} we obtain
\begin{align*}
\|z(t)-\tilde{z}(t)\|_2 &\leq \|\tL - \LV\|_2 t e^{t\|\tL\|_2} e^{t\|\LV\|_2} \\&= t \|\tL-\LV\|_2 e^{t\left(\|\LV\|_2 + \|\tL\|_2\right)}.
\end{align*}
We compute
\begin{align*}
\mathbb{P}\left(\|z(t)-\right.&\left.\tilde{z}(t)\|_2 \leq \varepsilon\right) \\&\geq \mathbb{P}\left(t \|\tL-\LV\|_2 e^{t\|\LV\|_2} e^{t\|\tL\|_2}\|z_0\|\leq \varepsilon\right)\\
&\geq  \mathbb{P}\left(t \|\tL-\LV\|_2 e^{2t\|\LV\|_2} e^{t\|\tL-\LV\|_2}\|z_0\|\leq \varepsilon\right)
\\&\geq \mathbb{P}\left(T \|\tL-\LV\|_2 e^{2T\|\LV\|_2} e^{T\|\tL-\LV\|_2}\|z_0\|\leq \varepsilon\right)
\end{align*}
By Theorem~\ref{t:generatorestimate} and $\|\cdot\|_2\leq \|\cdot\|_F$, for any $\tilde{\varepsilon}$ we can choose $m_0$ such that $\mathbb{P}\left( \|\tL - \LV\|_2 \leq \tilde{\varepsilon}\right) \geq 1-\delta$. Hence, there is $m_0$ only depending on $T$, $z_0$, $\LV$ and $\varepsilon$ such that for any $t\geq 0$
$$\mathbb{P}\left(\|z(t)-\tilde{z}(t)\|_2 \leq \varepsilon\right) \geq 1-\delta. $$
Taking the minimum over all $t\in [0,T]$ proves the claim.
\end{myproof}

\begin{myproof}[Corollary~\ref{c:control_trajectory}]
This proof follows with obvious modifications in the proof of Corollary~\ref{lem:gronwall} using the bound on then error of the time dependent generators of Theorem~\ref{t:coupled_estimate}.
\end{myproof}

\begin{corollary}\label{c:refinements}
If the Koopman semigroup generated by $\LV$ is bounded by $M$, then 
\begin{align*}
\|\tilde{z}(t)-z(t)\|_2 \leq M\|\tL - \LV\|_2 \|\tilde{z}\|_{L^1(0,t;\mathbb{R}^N)}.
\end{align*}
If it is exponentially stable then
\begin{align*}
\|\tilde{z}(t)-z(t)\|_2 \leq M   c\|\tL - \LV\|_2 \|\tilde{z}\|_{L^p(0,t;\mathbb{R}^N)}
\end{align*}
for any $1\leq p\leq \infty$ with $M\geq 1$ and $c=c(p)\geq 0$ independent of $t$. If additionally the semigroup generated by $\tL$ is exponentially stable, $\|\tilde{z}(t)-z(t)\|_2$ can be bounded uniformly in $t \geq 0$.
\end{corollary}
\begin{proof}
Subtracting \eqref{e:tz} from \eqref{e:z} and denoting $e(t)=\tilde{z}(t)-z(t)$ yields the system
\begin{align*}
\dot{e}(t) = \LV e(t)  + (\tL-\LV)\tilde{z}(t).
\end{align*}
Denoting by $\mathcal{K}^t_\mathbb{V}$ the Koopman semigroup generated by $\LV$ yields, using the variation of constants formula
\begin{align*}
e(t) &= \int_0^t \mathcal{K}^{t-s}_\mathbb{V} \left(\tL - \LV\right)\tilde{z}(s)\,\text{d}s
\end{align*}
and hence
\begin{align*}
\|e(t)\|_2 \leq \int_0^t \|\mathcal{K}^{t-s}_\mathbb{V}\|_2 \|\tL - \LV\|_2 \|\tilde{z}(s)\|_2\,\text{d}s.
\end{align*}
If $\mathcal{K}^t_\mathbb{V}$ is bounded by $M$, i.e., $\|\mathcal{K}^t_\mathbb{V}\|\leq M$, we have
\begin{align*}
\|e(t)\|_2 \leq M  \|\tL - \LV\|_2 \|\tilde{z}\|_{L^1(0,t;\mathbb{R}^N)}.
\end{align*}
If $\mathcal{K}^t_\mathbb{V}$ is exponentially stable, i.e., $\|\mathcal{K}^t_\mathbb{V}\|_2\leq Me^{-\omega t}$, we obtain
\begin{align*}
\|e(t)\|_2 \leq M c\|\tL - \LV\|_2 \|\tilde{z}\|_{L^p(0,t;\mathbb{R}^N)}
\end{align*}
for any $1\leq p\leq \infty$ with $c=c(p,\omega)$. If additionally, the semigroup generated by $\tL$ is exponentially stable implying that $\|\tilde{z}(t)\|_2\leq \tilde{M}e^{-\tilde{\omega}t}\|{z}_0\|_2$, this upper bound can be bounded uniformly in $t$.
\end{proof}

\MS{\subsection{Analytical Expressions for the OU Process}
\label{app:ou_analytical}
For the one-dimensional SDE~\eqref{eq:ou_sde}, the Koopman generator is given by:
\begin{align*}
\mathcal{L}\phi &= - x \phi'(x) + \frac{1}{2} \phi''(x).
\end{align*}
The eigenvalues of the generator are given by negative integers $\kappa_l = -l$, eigenvalues of the Koopman operator are their exponentials, as usual, $\lambda_l(t) = e^{- l t}$. The corresponding eigenfunctions are given by scaled physicist's Hermite polynomials. They are orthonormal with respect to the inner product with weight function $\mu$, which is the density of a normal distribution with variance one half, yielding the relations:
\begin{align}
\label{eq:ortho_hermite_poly}
\mu(x) &= \frac{1}{\sqrt{\pi}} \exp(-x^2), & \psi_l &= \frac{1}{\sqrt{2^l (l-1)!}} H_l(x), &\langle H_l, \,H_m\rangle_\mu &= 2^l l! \delta_{lm}.
\end{align}
The monomial basis can be recovered from eigenfunction basis $\psi_i$ by the representation formula:
\begin{align}
\label{eq:expansion_monomials}
x^n &= \frac{n!}{2^n}\sum_{k=0}^{\floor{\frac{n}{2}}} \frac{1}{k! (n - 2k)!}H_{n - 2k}(x).
\end{align}
For a basis set comprised of monomials up to maximal degree $N$, the Galerkin matrices $C$ and $A$ can be obtained as the moments of the normal distribution with variance $0.5$:
\begin{align*}
C_{ij} &= \begin{cases}
\frac{(i+j)!}{2^{i+j} ((i+j)/2)!} & (i+j) \, \text{even}, \\
0 & (i+j) \, \text{odd},
\end{cases} &
A_{ij} &= \begin{cases}
-\frac{ij}{2} \frac{(i+j-2)!}{2^{i+j-2} ((i+j-2)/2)!} & (i+j) \, \text{even}, \\
0 & (i+j) \, \text{odd}.
\end{cases} &
\end{align*}

\noindent For their numerical estimation, we consider centered random variables:
\begin{align*}
\phi_{ij}(x) &= x^i\, x^j - C_{ij} \quad\text{for }C, & \phi_{ij}(x) &= -\frac{ij}{2} x^{i-1}\, x^{j-1} - A_{ij} \quad\text{for }A.
\end{align*}
We calculate the asymptotic variance of the scalar random variable $\phi_{ij}$ if it is defined by either of the two expressions above. We also introduce the quantity $n := i+j$ for C or $n := i+j-2$ for A. The analytical expressions for $C_{ij},\, A_{ij}$ above exactly equal the terms corresponding to $H_0$ in the general expansion for the monomial $x^n$ in~\eqref{eq:expansion_monomials}. As the random variables $\phi_{ij}$ are centered, no contribution from $H_0$ is left. Thereby, we obtain the decomposition for $\phi_{ij}$ (up to the factor $-\frac{ij}{2}$ for estimation of $A$):
\begin{align}
\label{eq:expansion_phi}
\phi_{ij} &= x^n - \mathbb{E}^\mu[x^n] = \frac{n!}{2^n}\sum_{k=0}^{\ceil{\frac{n}{2} - 1}} \frac{1}{k! (n - 2k)!}H_{n - 2k}(x).
\end{align}
Next, we calculate matrix elements with the Koopman operator at lag time $l\Delta_t$ by combining~\eqref{eq:expansion_phi} with the orthogonality relation~\eqref{eq:ortho_hermite_poly}:
\begin{align*}
\langle \phi_{ij}, \, \mathcal{K}^{l\Delta_t}\phi_{ij}\rangle_\mu &= \frac{(n!)^2}{2^{2n}}\sum_{k=0}^{\ceil{\frac{n}{2} - 1}} \frac{1}{(k! (n - 2k)!)^2} e^{-(n-2k)l\Delta_t} 2^{n-2k} (n- 2k)! \\
&= \frac{(n!)^2}{2^{2n}}\sum_{k=0}^{\ceil{\frac{n}{2} - 1}} \frac{2^{n-2k}}{(k!)^2 (n - 2k)!} e^{-(n-2k)l\Delta_t}.
\end{align*}
Finally, by setting $q_k = e^{-(n - 2k)\Delta_t}$, we calculate the asymptotic variance according to the result in Lemma~\ref{lemma:co-variances} (note that the contribution for $l = 0$ appears only once, and that the result needs to be multiplied by $\frac{1}{4}ij$ for the estimation of $A$):
\begin{align*}
\sigma_{\phi_{ij}, \infty}^2 &= \langle \phi_{ij}, \,\phi_{ij}\rangle_\mu + 2 \sum_{l=1}^\infty \langle \phi_{ij}, \, \mathcal{K}^{l\Delta_t} \phi_{ij}\rangle_\mu \\
&= \frac{(n!)^2}{2^{2n}}\sum_{k=0}^{\ceil{\frac{n}{2} - 1}} \frac{2^{n-2k}}{(k!)^2 (n - 2k)!} \left[ \sum_{l=0}^\infty q_k^l + \sum_{l=1}^\infty q_k^l \right] \\
&= \frac{(n!)^2}{2^{2n}}\sum_{k=0}^{\ceil{\frac{n}{2} - 1}} \frac{2^{n-2k}}{(k!)^2 (n - 2k)!} \left[ \frac{1}{1 - q_k} + \frac{q_k}{1 - q_k} \right] \\
&= \frac{(n!)^2}{2^{2n}}\sum_{k=0}^{\ceil{\frac{n}{2} - 1}} \frac{2^{n-2k}}{(k!)^2 (n - 2k)!} \frac{1 + q_k}{1 - q_k}.
\end{align*}}

\end{document}